\documentclass[11pt]{article}
\usepackage{graphicx} 
\usepackage[margin=1in, a4paper]{geometry}
\usepackage{authblk}

\usepackage[T1]{fontenc}
\usepackage{lmodern}
\usepackage{amsthm, amsfonts,amsmath,amssymb,bbm, bbold}
\usepackage{mleftright}
\usepackage{verbatim}
\usepackage{textgreek}
\usepackage[]{authblk}
\usepackage{color}

\usepackage{etoolbox }

\makeatletter
\let\ams@starttoc\@starttoc
\makeatother
\usepackage[parfill]{parskip}
\makeatletter
\let\@starttoc\ams@starttoc
\patchcmd{\@starttoc}{\makeatletter}{\makeatletter\parskip\z@}{}{}
\makeatother

\usepackage{enumerate}
\usepackage{pdfsync}
\usepackage{color}

\numberwithin{equation}{section}

  \newtheorem{theorem}{Theorem}[section]
  \newtheorem{proposition}[theorem]{Proposition}
  
   \newtheorem{corollary}[theorem]{Corollary}

\theoremstyle{definition}

\theoremstyle{remark}
\newtheorem{remark}[theorem]{Remark}

\newcommand{\lb}{\mleft(}
\newcommand{\rb}{\mright)}
\newcommand{\lbb}{\mleft [}
\newcommand{\rbb}{\mright ]}
\newcommand{\labs}{\mleft |}
\newcommand{\rabs}{\mright |}
\makeatletter
\DeclareRobustCommand{\lbrb}[1]{%
  \ifmmode
    \ifinner
      (#1)%
    \else
      \mleft( #1 \mright)%
    \fi
  \else
    (#1)%
  \fi
}
\makeatother
\newcommand{\lbbrbb}[1]{\lbb#1\rbb}

\newcommand{\lbbrb}[1]{\lbb#1\rb}
\newcommand{\lbrbb}[1]{\lb#1\rbb}

\newcommand{\lbcurly}{\mleft\{}
\newcommand{\rbcurly}{\mright\}}
\newcommand{\lbcurlyrbcurly}[1]{\lbcurly#1\rbcurly}

\newcommand{\simi}{\stackrel{\infty}{\sim}}

\newcommand{\simo}{\stackrel{0}{\sim}}

\newcommand{\intervalOI}{\lbrb{0,\infty}}
\newcommand{\intervalOOI}{\lbbrb{0,\infty}}
\newcommand{\abs}[1]{\labs#1\rabs}

\newcommand{\curly}[1]{\lbcurly#1\rbcurly}

\newcommand{\bo}[1]{\mathrm{O}\lbrb{#1}}
\newcommand{\so}[1]{\mathrm{o}\lbrb{#1}}

\newcommand{\Pbb}[1]{\Pb\lb #1\rb}
\newcommand{\Ebb}[1]{\Eb\lbb #1\rbb}

\newcommand{\LL}{L\'{e}vy }
\newcommand{\BG}{Bernstein\allowbreak-gamma }
\newcommand{\LLP}{L\'{e}vy process }
\newcommand{\LLPs}{L\'{e}vy processes }
\newcommand{\LLK}{\LL\!\!-Khintchine }

\newcommand{\limi}[1]{\lim_{#1\to \infty}}
\newcommand{\limsupi}[1]{\limsup_{#1\to \infty}}
\newcommand{\liminfi}[1]{\liminf_{#1\to \infty}}
\newcommand{\limo}[1]{\lim_{#1\to 0}}

\newcommand{\limsupo}[1]{\limsup_{#1\to 0}}

\newcommand{\Cb}{\mathbb{C}}

\newcommand{\Eb}{\mathbb{E}}

\newcommand{\Nb}{\mathbb{N}}

\newcommand{\Rb}{\mathbb{R}}

\newcommand{\Rp}{\mathbb{R}^+}

\newcommand{\Pb}{\mathbb{P}}

\newcommand{\Zb}{\mathbb{Z}}

\newcommand{\dbf}{\mathbf{d}}

\newcommand{\Ac}{\mathcal{A}}
\newcommand{\Bc}{\mathcal{B}}

\newcommand{\Gc}{\mathcal{G}}

\newcommand{\Mcc}{\mathcal{M}}

\newcommand{\Zc}{\mathcal{Z}}

\newcommand{\supp}{{\rm{Supp}}}

\newcommand{\Att}{\mathtt{A}}

\newcommand{\Mtt}{\mathtt{M}}

\newcommand{\ak}{\mathfrak{a}}

\newcommand{\uk}{\mathfrak{u}}

\newcommand{\ind}[1]{\mathbbm{1}_{\lbcurlyrbcurly{#1}}}

\newcommand{\IntOI}{\int_{0}^{\infty}}

\newcommand{\IntII}{\int_{-\infty}^{\infty}}

\newcommand{\ccoRp}{\mathtt{C}_0\lbrb{\Rb^+}}

\newcommand{\Aph}{A_\phi}
\newcommand{\aph}{\mathfrak{a}_\phi}
\newcommand{\baph}{\bar{\mathfrak{a}}_\phi}

\newcommand{\uph}{\mathfrak{u}_{\phi}}

\newcommand{\Wp}{W_\phi}
\newcommand{\Wpp}{W_{\phi_+}}
\newcommand{\Wpn}{W_{\phi_-}}

\newcommand{\Wphi}{W_\phi}

\newcommand{\Yp}{Y_\phi}

\newcommand{\pPs}{f_\Psi}
\newcommand{\PIp}{F_\Psi}
\newcommand{\PIPs}[1]{\overline{F}_{\Psi}(#1)}

\newcommand{\CbOI}{\Cb_{\intervalOI}}
\newcommand{\CbOOI}{\Cb_{\intervalOOI}}

\newcommand{\IPsi}{I_\Psi}

\newcommand{\MPsi}{\Mcc_{I_\Psi}}

\newcommand{\NPs}{\mathtt{N}_{\Psi}}

\usepackage{pdflscape}    
\usepackage{longtable}

\usepackage{authblk}

\newcommand{\D}{\mathrm{d}}

\makeatletter
\newcommand{\myitem}[1]{%
\item[#1]\protected@edef\@currentlabel{#1}%
}
\makeatother

\usepackage{tikz}
\usetikzlibrary{decorations.markings}
\usetikzlibrary{patterns}
\DeclareMathSymbol{\mrq}{\mathord}{operators}{`'}
\usepackage{wrapfig}

\allowbreak 

\usepackage[table]{xcolor}        
\definecolor{magenta}{rgb}{1,0,1} 

\usepackage{hyperref}
\usepackage{array}
\usepackage{booktabs}           
\usepackage{capt-of}
\usepackage{multirow}
\newcommand{\softline}{%
  \arrayrulecolor{black!18}\hline\arrayrulecolor{black}%
}
\newcommand{\softcline}{%
  \arrayrulecolor{black!18}\cline{2-3}\arrayrulecolor{black}%
}
\newcommand{\blockline}{%
  \arrayrulecolor{black!40}\hline\arrayrulecolor{black}%
}

\newcommand{\Ebbt}[1]{\mathbb{E}[#1]}

\newcolumntype{P}[1]{>{\centering$\displaystyle}p{#1}<{$}}

\usepackage[intoc]{nomencl}
\makenomenclature
\usepackage{etoolbox}
\renewcommand{\nomgroup}[1]{%
  \item[\bfseries
    \ifstrequal{#1}{A}{General}%
      {%
        \ifstrequal{#1}{B}{Lévy processes and exponential functionals}%
          {%
            \ifstrequal{#1}{C}{Bernstein-gamma functions}%
              {}
          }%
      }%
  ]%
}

\newcommand{\ab}{a+ib}

\renewcommand{\Ac}{\mathtt{A}}
\renewcommand{\Im}{\mathtt{Im}}
\renewcommand{\Re}{\mathtt{Re}}

\title{Recent developments in exponential functionals of Lévy processes}
\date{}  

\author[1,2]{Martin Minchev\thanks{Email: \texttt{martin.minchev@math.uzh.ch}}}
\author[1,3]{Mladen Savov\thanks{Email: \texttt{msavov@fmi.uni-sofia.bg}}}

\affil[1]{Faculty of Mathematics and Informatics, Sofia University “St.\ Kliment Ohridski”, 
   5 James Bourchier Blvd., 1164 Sofia, Bulgaria}
\affil[2]{Institute of Mathematics, University of Zurich, Winterthurerstrasse 190, 
8057 Zürich, Switzerland}
\affil[3]{Institute of Mathematics and Informatics, Bulgarian Academy of Sciences, 8 Akad. Georgi Bonchev Str., 1113, Sofia, Bulgaria}

\begin{document}


\bigskip

\maketitle

\begin{abstract}
    This survey aims to review two decades of progress on exponential functionals of (possibly killed) real-valued Lévy processes. Since the publication of the seminal survey by Bertoin and Yor~\cite{BerYor05}, substantial advances have been made in understanding the structure and properties of these random variables. At the same time, numerous applications of these quantities have emerged across various different contexts of modern applied probability. Motivated by all this, in this manuscript, we provide a detailed overview of these developments, beginning with a discussion of the class of special functions that have played a central role in recent progress, and then organising the main results on exponential functionals into thematic groups. Moreover, we complement several of these results and set them within a unified framework. Throughout, we strive to offer a coherent historical account of each contribution, highlighting both the probabilistic and analytical techniques that have driven the advances in the field.
\end{abstract}

\tableofcontents
\section{Introduction}
In this survey, we review two decades of progress in the study of exponential functionals of (possibly killed) real-valued Lévy processes, following the classical treatment by Bertoin and Yor \cite{BerYor05}.

We recall that an unkilled \LL process $\xi=(\xi_s)_{s\geq 0}$ possesses almost surely (a.s.) right-continuous paths, starts from zero (i.e., $\xi_0=0$), and satisfies that, for any $t>0$, in law
\[
   (\xi_{t+s}-\xi_t)_{s\geq 0}\;\stackrel{w}{=}\;(\xi_s)_{s\geq 0},
\]
with $(\xi_{t+s}-\xi_t)_{s\geq 0}$ independent of $(\xi_s)_{s\leq t}$. The latter two properties are structural and are known as stationary and independent increments. Potential killing is introduced via an exponential time $\mathbf{e}_q \sim \mathrm{Exp}(q)$, $q \geq 0$, 
which is independent of $\xi$, and such that $\xi_t = \infty$ for $t \geq \mathbf{e}_q$, 
with $\mathbf{e}_0 \equiv \infty$. Potentially killed \LLPs are in bijection with negative definite functions via the relation
\begin{equation}\label{LLK}
    \Psi(z) \;=\; -q + \gamma z + \tfrac{\sigma^2}{2} z^2 
    + \int_{\mathbb{R}} \!\big(e^{zy} - 1 - \ind{|y|\leq 1}\,zy\big)\,\Pi(\D y),
    \qquad z \in i\Rb,
\end{equation}
where $\gamma \in \Rb$, $\sigma^2 \geq 0$, $q \geq 0$, and $\Pi$ is a $\sigma$-finite measure satisfying the integrability condition 
\[
   \int_{\mathbb{R}} \min\{y^2,1\}\,\Pi(\D y) < \infty.
\] The quantities are known respectively as: the linear term (\(\gamma\)), the diffusion term (\(\sigma^2\)), the killing term (\(q\)), and the \LL measure (\(\Pi\)). The measure \(\Pi\) determines the intensity and the spatial structure of the jumps of \(\xi\). The trivial case of unkilled \LLP satisfying \(\xi\equiv C\) is excluded from our considerations. \textbf{In this survey by \LLPs we understand potentially killed \LL processes.}

This paper surveys and occasionally complements the results on the  quantity
\nomenclature[B]{$\Psi$}{\LLK exponent of a Lévy process $\xi$, \eqref{Npsi}}
\nomenclature[B]{$\MPsi(z)$}{Mellin transform of $I_\Psi$, \eqref{MT}}
\nomenclature[B]{$\NPs$}{decay of $\MPsi$ along complex lines, Theorem \ref{MDecay}, \eqref{Npsi}}
\nomenclature[B]{\(\dbf_\pm\) }{drift terms of $\phi_\pm$}
\nomenclature[B]{\(F_\Psi, \, f_\psi\) }{c.d.f. and density of $I_\Psi$, Section \ref{SmoothSec}}
\nomenclature[B]{$I_\Psi$}{exponential functional of $\xi$, \eqref{ExpFunc}}
\nomenclature[B]{$\phi_\pm$}{Wiener-Hopf factors from $\Psi(z)=-\phi_+(-z)\phi_-(z)$, \eqref{WHf}}
\begin{equation}\label{ExpFunc}
    \begin{split}
        & \IPsi := \IntOI e^{-\xi_s}\D s = \int_0^{\mathbf{e}_q} e^{-\xi_s}\D s 
        .
    \end{split}
\end{equation}
Our approach is to present various groups of results related to exponential functionals, such as the finiteness and computation of moments, smoothness of densities, and more, whilst also outlining their historical development and where suitable discuss the main tools used in their derivation. We also dedicate a section to the properties of the so-called Bernstein-gamma functions, which have proven to be central in the most recent developments in the area.
\subsection{Literature review}
Here we provide a brief literature review on the development of the theory of exponential functionals of \LL processes. In our view, the recurrence equation~\eqref{eq:recurrence} has been the most instrumental tool for deriving a few of the sharpest results in this area. Nevertheless, many important contributions have also been obtained through alternative approaches and methodologies.

The study of exponential functionals was initiated by Urbanik in the broader context of infinite divisibility and multiplicative infinite divisibility of functionals of stochastic processes; see~\cite{Urban95}. Subsequent work on general properties of exponential functionals of \LL processes, including the existence of moments, integral equations for the law, and special cases of the recurrence equation~\eqref{eq:recurrence}, crucial in their study, was carried out mainly by M.~Yor and co-authors; see \cite{BerYor05,CarPetYor94,CarPetYor01,Yor01}. Of particular note is the work of Zwart and co-authors, who established a relatively general version of the recurrence equation for the moments and applied it in special cases; see \cite{GuRoZw04,Maulik-Zwart-06}, where, for instance, they also derived the density of the exponential functional for non-decreasing compound Poisson processes. Following \cite{Maulik-Zwart-06}, it became increasingly clear that solving the recurrence equation for the moments can identify not only the moments but also the law of the exponential functional of \LL processes. This insight led to the derivation of the Mellin transform of the exponential functional for various classes of \LLPs and, in certain cases, to series representations of their law; see \cite{Kuz2012,Kuznetsov13,KuKyPa12,KuzPar13}. Finally, the recurrence equation was solved in full generality in terms of the so-called Bernstein-gamma functions in \cite{PatieSavov2018}. This work culminated earlier studies on factorisations of the law of the exponential functional; see \cite{PardoPatieSavov2012,PatieSavov2012,PatieSavov2013}. An equivalent form of~\eqref{eq:recurrence}, involving the density of the exponential functional and the potential measure of the underlying \LL process, was presented and applied in \cite{ArRiv23}.

 Some results concerning the law of the exponential functional have been obtained by exploiting the fact that this random variable satisfies a random affine equation --- a relation extensively studied in the literature; see, e.g., \cite{ArRiv23,Maulik-Zwart-06,Riv05,Rivero12}. Through its connection with positive self-similar Markov processes, the law of the exponential functional of spectrally positive \LLPs has been expressed in series form; see \cite{Patie2012}. Moreover, the fact that the exponential functional arises as the stationary law of Ornstein--Uhlenbeck processes has been used to derive factorisations of its law; see \cite{PardoPatieSavov2012}.

The
 main reason for the interest in these random variables is due to the fact that exponential
 functionals have played a role in various domains of theoretical and
 applied probability such as the spectral theory of some non-reversible
 Markov semigroups (\cite{PatSav17,PatSav21,PaSaZh19,Savov17}),
 the study of random planar maps (\cite{BerCurKor-18,Budd18}),
 limiting theorems for Markov chains (\cite{BerKor16}), positive
 self-similar Markov processes
 (\cite{Bertoin-Caballero-02,Bertoin2002,CabCha06,Caballero2006,Patie-2009,Patie-06c}), financial and insurance mathematics
 (\cite{HackKuz14,KluLinMal06,Patie-As}), branching processes with
 immigration (\cite{Patie-2009}), random processes in random environment
 (\cite{BoiHut12,LiXu18,PalParSma16, Xu_2021}), random fragmentations (\cite{Ber2002,Bertoin2006,Dadoun2017,Haas23,Stephenson2018}) with \cite{Stephenson2018} dealing with multitype fragmentation, extinction times of
 self-similar processes (\cite{LoePatSav19}), stationary distributions of Lévy-driven Ornstein--Uhlenbeck processes (\cite{Behme2015,BehLin15,BehLinMal_11,BehLinRek21,BeLiReRi21,Kevei2018}) and others; see \cite{GouLu25}.  Recently exponential functionals appeared in the study of general self-similar Markov trees; see \cite{BerCurRie24}.

Let \(\IPsi(t)=\int_{0}^t e^{-\xi_s}\,\D s,\; t\geq 0,\) be the exponential functional on deterministic horizon. The law of \(\IPsi(t)\) is considerably harder to study. Nevertheless, some recent progress relies on the observation that the classical exponential functional of a killed \LL process \(I_{\Psi_q}\), with \(\Psi_q(z)=\Psi(z)-q\), is related to the Laplace transform of \(\IPsi(t)\). Consequently, the representation of the moments of \(I_{\Psi_q}\) in \eqref{MPsiAn}, combined with Laplace and Mellin inversion, has yielded useful results on the moments of \(\IPsi(t)\); see \cite{Barker19,BarkerSavov2021,PalSarSav24} for series expansions of the moments when the \LLP is a subordinator, and \cite{PalSarSav24} for explicit evaluations of some moments when \(\xi\) is symmetric as well as for general convolutional equations. Moreover, using the same link, the finiteness of integrals of the type \(\int_0^\infty t^n F(\IPsi(t))\,\D t\), for suitable functions \(F\) and integer \(n\), and scaling limits for \(\Pbb{\IPsi(t)\in \D x}\), have been studied in \cite{Minchev17,Minchev-Savov-2025}.  We highlight that some papers in the short literature review on exponential functionals above, in fact, use exponential functionals on a deterministic horizon, but the properties of the latter are derived with the help of the former.

\subsection{Common notation}


We have included an index of notation at the end of the survey. However, we spell out some of it now to enable smooth reading. We denote by \(\Nb\) and \(\Zb\) the sets of non-negative integers and all integers, respectively. We set \(\Rb := (-\infty, \infty)\) for the real line, \(\Rp := ({0, \infty})\) for the positive real half-line, and \(\Cb\) for the complex plane. For \(z\in\Cb\) we use \(\Re(z)\) for its real part and \(\Im(z)\) for its imaginary part.
Next, for \(-\infty \leq a < b \leq \infty\), we set \(\Cb_{(a, b)} := \curly{z \in \Cb : \Re(z) \in \lbrb{a, b}}\) for a complex strip, and use \(\Cb_a := \curly{z \in \Cb : \Re(z) = a}\), \(a \in \Rb\), for a complex line. For the same range of \(a\) and \(b\), we denote by \(\Att_{(a, b)}\) (respectively, \(\Mtt_{(a, b)}\)) the space of functions holomorphic (respectively, meromorphic) on \(\Cb_{(a, b)}\).
The space \(\Att_{[a, b)}\), with \(-\infty < a < b \leq \infty\), consists of those functions in \(\Att_{(a, b)}\) that can be continuously extended to \(\Cb_a\); other variations such as \(\Att_{[a, b]}, \Att_{(a, b]}\) are defined analogously. Throughout we use \(\log_0:\Cb\setminus\lbrbb{-\infty,0}\to \Cb\) for the main branch of the complex logarithm defined via the main branch of the argument \(\arg_0: \Cb\to \lbrbb{-\pi,\pi}\). When we do not specify the branch we use \(\log\).

\nomenclature[A]{\(\Cb_{a}\)}{
  $\curly{z \in \Cb : \Re(z)=a}$}
  \nomenclature[A]{\(\Cb_{(a, b)}\)}{
  $\curly{z \in \Cb : \Re(z) \in \lbrb{a, b}}$}
    \nomenclature[A]{\(\Att_{(a, b)},\, \Mtt_{(a, b)}\)}{
  analytic and, respectively, meromorphic functions on
   \(\Cb_{(a, b)}\)
  }
      \nomenclature[A]{\(\log_0,\, \log\)}{
 principal and, respectively, some branch of the complex logarithm
  }

\section{Bernstein-gamma functions}\label{sec:BG}

A natural object in the description of exponential functionals are \textit{Bernstein-gamma functions}, which are defined as the unique solution to
\begin{equation}\label{receq}
    \Wphi(z+1) = \phi(z)\Wphi(z),\quad 
    \text{on}
    \quad \Re(z) > 0;\quad 
    \text{with}
    \quad\Wphi(1) = 1,
\end{equation}
where $\phi$ is a Bernstein function, the class of which we recall below.
It is further known that there exists a random variable \(Y_\phi\) such that
 \(\Pbb{\Yp > 0} = 1\) and \(\Wphi(z) = \Ebb{\Yp^{z-1}} 
\) for  \(\Re(z) > 0\), see \cite[Thm.~6.1]{PatSav21} and \cite[Sec.~4]{PatieSavov2018}. The relevance of these special functions in the study of \(I_\Psi\) stems from the possibility of extending the following recurrence equation \begin{equation}\label{eq:recurrence}
    \Ebb{I_\Psi^{z}} \;=\; \frac{-z}{\Psi(-z)}\,\Ebb{I_\Psi^{\,z-1}},
\end{equation} to the complex plane, or to suitable strips therein.
This equation has already appeared in certain special cases in~\cite{CarPetYor97}, and its subsequent development is discussed in Remark~\ref{MpsiH}. Then, using the Wiener-Hopf factorisation of \(\Psi(z)\) as \(-\phi_+(-z)\phi_-(z)\), described in Section~\ref{sec: Mellin transform and moments}, one may formally express the solution to \eqref{eq:recurrence}, that is, the complex moments of \(I_\Psi\), as
\[
   \Ebb{I_\Psi^{\,z-1}} \;=\; \phi_-(0)\,\frac{\Gamma(z)}{W_{\phi_+}(z)}\,W_{\phi_-}(1-z).
\]
The precise statement, due to~\cite{PatieSavov2018}, is presented in Theorem~\ref{thm: MPsi}.

Having outlined the importance of Bernstein-gamma functions in the context of this survey, we now present some of their basic properties and characteristics, while more intricate complex-analytical aspects are deferred to Section~\ref{sec: BG further}. We also emphasise that Bernstein-gamma functions are of independent interest, as they subsume several well-known special functions, starting with Euler’s Gamma function~\cite{Euler1738}, and later developments such as the Barnes gamma functions and the $q$-gamma function from $q$-calculus; see~\cite{BeBiYo04,Chhaibi16}. Moreover, \cite{PatSav21} demonstrates that Bernstein-gamma functions play a fundamental role in the study of Markovian self-similarity.

We recall that a function \(\phi:\Cb_{\lbbrb{0,\infty}}\to \Cb\) is called a \textit{Bernstein function} if and only if
\begin{equation}\label{BG}
  \phi(z) \;=\; \phi(0)+\dbf z+\IntOI \lbrb{1-e^{-zy}}\,\mu(\D y) 
  \;=\; \phi(0)+\dbf z+z\IntOI e^{-zy}\,\bar{\mu}(y)\,\D y,
\end{equation}
where \(\phi(0)\) and \(\dbf\) are non-negative constants, $\mu$ is a \(\sigma\)-finite measure with 
\(\IntOI\min\curly{y,1}\,\mu(\D y)<\infty\), and 
\[
   \bar\mu(y):=\mu\lbrb{\lbrb{y,\infty}}, \qquad y\geq 0.
\]
Bernstein functions are in bijection with subordinators (i.e., non-decreasing \LL processes) via
\begin{equation}\label{B-Sub}
    \phi(z):=-\log_0 \Ebb{e^{-z\xi_1}}, \qquad z\in\Cb_{\lbbrb{0,\infty}},
\end{equation}
where $\xi$ is a subordinator. Up to an additive constant, they are related to completely monotone functions through their derivatives, namely
\begin{equation}\label{B-CM}
    \phi'(z) \;=\; \dbf + \IntOI  e^{-zy}\,y\,\mu(\D y)
    \;=\; \IntOI e^{-zy}\,\lbrb{y\mu(\D y)+\dbf\delta_0(\D y)}, 
    \qquad z\in\Cb_{\lbbrb{0,\infty}}.
\end{equation}

In this survey, we introduce only those quantities and properties pertaining to Bernstein functions that are needed here, and we refer to~\cite{BarkerSavov2021,MinSav_2023,SchVon2020} for extensive results concerning them. 
Note that always $\phi \in \Att_{\intervalOOI}$, and set
\begin{align}\label{aph}
    \aph &:= \inf\curly{u<0:\, \phi\in \Att_{(u,\infty)}} 
        = -\sup\curly{\lambda\geq0:\, \Ebb{e^{\lambda \xi_1}} \text{ is finite}} 
        \in \lbbrbb{-\infty,0}, \\
    \uph &:= \sup\curly{u\in\lbbrbb{\aph,0}:\, \phi(u)=0} 
        \in \lbbrbb{-\infty,0}, \nonumber \\
    \baph &:= \max\curly{\aph,\uph} 
        \in \lbbrbb{-\infty,0}, \nonumber
\end{align}
with \(\sup \emptyset = -\infty\) and \(\inf \emptyset = 0\). By definition, if \(u_\phi\) is finite, it is not smaller than \(\aph\), so
\begin{equation}\label{bara=a}
    \text{if } \uph > -\infty, \quad \text{then } \baph = \uph,
\end{equation}
whereas it may happen that \(\uph = -\infty < \aph\). Take, for example, \(\phi(z) = (z+1)^\alpha + 1\) with \(\alpha \in \lbrb{0,1}\), in which case \(\phi\) is a Bernstein function with \(\aph = -1\) and \(\uph = -\infty\). Further examples of Bernstein functions, together with the corresponding quantities from~\eqref{aph}, are listed in Appendix~\ref{sec: BG examples}.
\nomenclature[C]{$\aph$}{%
  $\inf\curly{u<0: \phi\in \Att_{(u,\infty)}}\in[-\infty,0]$,    
  ~\eqref{aph},
}
\nomenclature[C]{$\uph$}{%
  $\sup\curly{u\in\lbbrbb{\aph,0}:\phi(u)=0} \in \lbbrbb{-\infty,0} $,    
  ~\eqref{aph},
}
\nomenclature[C]{$\baph$}{%
  $\max\curly{\aph,\uph} \in \lbbrbb{-\infty,0}$,    
  ~\eqref{aph},
}

Each Bernstein function is also related to the potential measure $U(\D x)$ of the associated subordinator via the classical Laplace transform identity
\begin{equation}\label{poten}
    \frac{1}{\phi(z)} \;=\; \IntOI e^{-zy}\,U(\D y) 
    \;=\; \IntOI e^{-zy}\,\IntOI \Pbb{\xi_t \in \D y}\,\D t,
    \qquad z \in \CbOI.
\end{equation}
Note that \(\phi(0)\) may be positive, in which case the potential measure coincides with the \(\phi(0)\)-potential measure of the unkilled subordinator pertaining to \(\phi(z)-\phi(0)\).

Bernstein-gamma functions, defined via~\eqref{receq}, inherit key features of the celebrated gamma function, which solves~\eqref{receq} with \(\phi(z)=z\). Below, we collect their most important properties, present illustrative examples, and provide a historical overview in the corresponding remarks.

We start by presenting the product and integral representations for \(\Wphi\), which extend well-known representations of the  gamma function itself.
\begin{theorem}[Thm.~6.0.1 in~\cite{PatSav21} and Thm.~4.7 in~\cite{PatieSavov2018}]\label{Wrep}
    Let \(\phi\) be a Bernstein function. Then we have the Weierstrass product representation
    \begin{equation}\label{WWrep}
        \Wphi(z) \;=\; \frac{e^{-\gamma_\phi z}}{\phi(z)}
        \prod_{k=1}^\infty \frac{\phi(k)}{\phi(k+z)} \,
        e^{\tfrac{\phi'(k)}{\phi(k)}z}
        \;\in\; \Att_{\lbrb{\baph,\infty}} \cap \Mtt_{\lbrb{\aph,\infty}},
    \end{equation}
    where 
    \begin{equation}\label{gmph}
        \gamma_\phi \;:=\; \limi{n}\left(\sum_{k=1}^n \frac{\phi'(k)}{\phi(k)} - \ln\phi(n)\right) 
        \;\in\; \lbbrbb{-\ln\phi(1), \, \tfrac{\phi'(1)}{\phi(1)} - \ln\phi(1)}.
    \end{equation}
    Moreover, we have the integral representation
    \begin{equation}\label{WWrep1}
        \log \Wphi(z+1)
        \;=\; z\ln\phi(1) + \IntOI \lbrb{e^{-zy}-1-z(e^{-y}-1)} 
        \frac{k(\D y)}{y(e^y-1)}, 
        \qquad z \in \Cb_{\lbrb{-1,\infty}},
    \end{equation}
    where \(k(\D y):=\int_0^y U(\D y-r)\lbrb{r\mu(\D r)+\dbf\delta_0(\D r)}\), 
    with \(U\) being the potential measure of the subordinator pertaining to \(\phi\).  
    Furthermore, the Laplace transform of \(k\) is
    \[
        \int_0^\infty e^{-\lambda y} k(\D y) 
        \;=\; \frac{\phi'(\lambda)}{\phi(\lambda)}, 
        \qquad \lambda \in \Rp.
    \]
\end{theorem}

\begin{remark}[\textit{\textbf{Historical remark}}]\label{WrepRH}
  The product representation~\eqref{WWrep} is classical for equations of the type 
  \(f(x+1)=g(x)f(x)\), \(x>0\), with \(g>0\) and \(\log\)-concave on \(\Rp\); 
  see~\cite[Thm.~7.1]{Web97}. When \(g=\phi\) for special classes of Bernstein functions \(\phi\), 
  relation~\eqref{WWrep} can be inferred from~\cite[Sec.~2]{Maulik-Zwart-06} 
  and~\cite[Prop.~7]{GuRoZw04}; in the general case,~\eqref{WWrep} is derived
  in~\cite[Thm. 2] {AliJedRiv14} and \cite[Thm. 2.3]{HirYor13}. The expression~\eqref{WWrep} was analytically 
  extended to \(z\in\Cb_{\lbrb{\baph,\infty}}\) in~\cite[Thm.~6.0.1]{PatSav21}.
  
  The integral representation~\eqref{WWrep1} is taken from~\cite[Thm.~2.2 and Thm.~3.1, respectively]{Berg07,HirYor13}. In~\cite[Thm.~4.7]{PatieSavov2018},  the term $\dbf\delta_0$ in~\eqref{WWrep1} is erroneously written as $\delta_\dbf$. In the case \(\phi(z)=z\), we have \(k(\D y)=U(\D y)=\D y\), and thereby recover the classical Malmstén formula for the gamma function.
\end{remark}

\begin{remark}[\textit{\textbf{Further comments}}]\label{WrepRF}
   When $\phi$ is a Wiener--Hopf factor of a \LLP $\xi$, the measure
$k(\D y)$ in~\eqref{WWrep1} can be represented in terms of the corresponding
harmonic renewal measure. For example, if $\phi$ is the Laplace exponent of
the ascending ladder height process of $\xi$, then on $\Rp$
\[
    H_+(\D y)
    :=
    \IntOI \frac{e^{-qt}}{t}\,\Pbb{\xi_t\in \D y}\,\D t
\]
is the harmonic renewal measure associated with the ascending Wiener--Hopf
factor, where $q\geq0$ is the killing rate of $\xi$ (which equals zero when
$\xi$ is conservative). In this case,
\[
    k(\D y)
    =
    yH_+(\D y)
    =
    y\IntOI \frac{e^{-qt}}{t}\,\Pbb{\xi_t\in \D y}\,\D t .
\]
The terminology reflects the harmonic weight $\D t/t$, which is the continuous
analogue of the harmonic sums appearing in the random-walk Wiener--Hopf
factorisation. This measure appears naturally in Fristedt's formula for the
ladder process; see, for example, \cite[Sec.~5.2]{Don07}. See also
\cite[(5.5)]{Minchev-Savov-2025} for a version extending to bivariate
Bernstein-gamma functions.
\end{remark}

We continue with a representation of \(\Wphi\) that is suitable for deriving asymptotics, for which we will need the following terms:
\begin{equation}\label{terms}
    \begin{split}
        L_\phi(z) &:= \int_{1 \to z+1} \log_0 \phi(\chi)\, d\chi, 
        \qquad z \in \Cb_{\lbrb{-1, \infty}}, 
        \\
        T_\phi &:= \frac{1}{2} \IntOI (u - \lfloor u \rfloor)(1 - (u - \lfloor u \rfloor)) 
        \left( \left( \frac{\phi'(u)}{\phi(u)} \right)^2 - \frac{\phi''(u)}{\phi(u)} \right) \D u 
        \;\in\; \Rb,
        \\
        E_\phi(z) &:= \frac{1}{2} \IntOI (u - \lfloor u \rfloor)(1 - (u - \lfloor u \rfloor)) 
        \left( \log_0\!\left( \frac{\phi(u+z)}{\phi(u)} \right)'' \right) \D u, 
        \qquad z \in \CbOI.
    \end{split}
\end{equation}
Here, the complex integral is taken along any contour in \(\Cb_{\lbrb{0, \infty}}\) connecting \(1\) to \(z+1\), \(\lfloor \cdot \rfloor\) denotes the floor function, and all derivatives inside the integrals are taken with respect to \(u\).
\nomenclature[C]{\( L_\phi, \, T_\phi,
\, E_\phi\,, \Aph\)}{components for the Stirling-type asymptotics for $W_\phi$, \eqref{terms}, \eqref{Aphi}}

\begin{theorem}[Thm.~2.9 in~\cite{BarkerSavov2021}]\label{repWTh}
    Let \(\phi\) be a Bernstein function that is not identically \(0\). Then
    \begin{equation}\label{repW}
        \Wphi(z) \;=\; \frac{1}{\phi(z)} 
        \sqrt{\frac{\phi(1)}{\phi(z+1)}} 
        e^{L_{\phi}(z)-E_\phi(z)}, 
        \qquad z \in \CbOI,
    \end{equation}
    with \(\sup_{z\in \CbOI}\abs{E_\phi(z)} \leq 2\).
\end{theorem}

\begin{remark}[\textit{\textbf{Historical remark}}]\label{repWThH}
    This representation, in a rather clumsy form, was first obtained in~\cite[Sec.~6]{PatSav21} and was subsequently improved in~\cite[Thm.~4.2]{PatieSavov2018}.
\end{remark}
\nomenclature[C]{$W_\phi$}{%
  Bernstein-gamma function of \(\phi\),  
  see Theorem~\ref{repWTh},  
  ~\eqref{repW}
}

\begin{remark}[\textit{\textbf{Further comments}}]\label{repWThF}
    Note that for \(z=a+ib\) with \(a>0\), integrating along \(1 \to 1+a \to 1+a+ib\) yields
    \begin{equation}\label{decomL}
        L_\phi(z) \;=\; \int_1^{1+a}\ln\phi(u)\,\D u
        + i \int_{0}^b \ln\abs{\phi(1+a+iu)}\,\D u
        - \int_{0}^b \arg_0\phi(1+a+iu)\,\D u.
    \end{equation}
   Therefore, the middle term does not contribute to the asymptotics of \(\abs{\Wp(z)}\).
    In~\cite{MinSav_2023,PatieSavov2018,PatSav21} the last integral is set to be \(A_\phi(z+1)\), and from the proof of~\cite[Thm.~3.2(1), p.~26]{PatieSavov2018},
    \begin{equation}\label{Aphi}
        A_\phi(z) \;:=\; \int_{0}^b \arg_0\phi(a+iu)\,\D u
        \;=\; \int_a^\infty \ln\abs{\frac{\phi(u+ib)}{\phi(u)}}\,\D u,
    \end{equation}
    which, since \(\abs{\phi(u+ib)} \geq \phi(u)\), shows that \(A_\phi(z)\) is non-increasing in \(a=\Re(z)\). 
\end{remark}

Next, we present the most general form of the Stirling-type asymptotics available for \(\Wphi\), 
whose proof is given in Section~\ref{sec: BG further}.
\begin{theorem}[Thm.~2.9 in~\cite{BarkerSavov2021} and Lem.~5.1 in~\cite{MinSav_2023}]\label{thm-WStir}
    Let \(\phi\) be a Bernstein function. Then, as \(|z|\to\infty\) with \(\Re(z)>0\),
    \begin{equation}\label{WStir}
        \Wphi(z) \;=\; \lbrb{\sqrt{\phi(1)}\,e^{-T_\phi} + \so{1}}
        \frac{e^{L_\phi(z-1)}}{\sqrt{\phi(z+1)}}\,,
    \end{equation}
    where the asymptotics are uniform if additionally \(\Re(z)\to\infty\). 
    Furthermore, for all \(a>0\),
    \begin{equation}\label{geomDecay}
        -\frac{\pi}{2} \;\leq\; \liminfi{|b|} \frac{\ln\abs{\Wphi(a+ib)}}{|b|} \;\leq\; 0,
    \end{equation}
    and
    \begin{equation}\label{Da}
        \text{for } D_a := \liminfi{|b|} \frac{\ln\abs{\Wphi(a+ib)}}{|b|}, 
        \quad \text{it holds that } D_a = D_1
        \text{ for all } a \in \Rp.
    \end{equation}
\end{theorem}
\begin{remark}[\textit{\textbf{Historical remark}}]\label{WStirRH}
    When \(z \to \infty\) along the real line, the asymptotics can be inferred from~\cite{Web97} and are also contained in~\cite[Thm.~5.1]{PatSav21}. 
    The first asymptotic results for \(\Wphi\) in the complex plane appear in~\cite[Thm.~6.0.2]{PatSav21}, 
    whereas~\cite[Thm.~4.2]{PatieSavov2018} contains the first general Stirling-type asymptotics for \(\Wphi\).
\end{remark}

\begin{remark}[\textit{\textbf{Further comments}}]\label{WStirRF}
    Due to the Mellin inversion formula, the asymptotic behaviour of \(\Wphi\) is of particular interest when \(z = a + ib\) with \(a > 0\) fixed. In this case,~\cite[Thm.~4.2]{PatieSavov2018} provides explicit asymptotics for \(\Wphi\) for various classes of Bernstein functions. In particular, if \(\dbf > 0\), then the first-order decay of \(\abs{\Wphi(a+ib)}\) is \(e^{-\frac{\pi}{2}|b|}\), which matches the lower bound in~\eqref{geomDecay} and coincides with the classical decay of the gamma function. More generally, the exponential decay is governed by
    \begin{equation}\label{geomDecay1}
        \ln\abs{\Wphi(a+ib)} \;\simi\; -|b|\,
        \frac{\int_0^{|b|} \arg_0\phi(a+iu)\,\D u}{|b|}
        \;\in\; \lbbrbb{-\tfrac{\pi}{2}|b|,\,0}.
    \end{equation}
    That is, for any \(a > 0\), the average argument (the geometry of \(\phi(\CbOOI)\)) of the mapping 
    \(\phi: \CbOOI \to \CbOOI\) determines the linear rate of decay at the logarithmic scale. 
    The fact that the gamma function provides a lower envelope for the absolute value of a \BG function is established in Proposition~\ref{prop:lowerboundW}.
    
    When applying the saddle point method in the Mellin inverse of quantities involving $\Wphi$, as in~\cite{MinSav_2023}, 
    the asymptotic behaviour of \(\Wphi\) along complex rays in \(\Cb_{\lbrb{0,\infty}}\) becomes relevant too. 
    In this setting, general results can be found in~\cite[Thm.~5.3]{MinSav_2023}.
\end{remark}

When \(z=a+ib\) with \(a\) fixed and \(|b|\to\infty\) we can use Theorem \ref{thm-WStir} and Remark \ref{repWThF} to get the neat asymptotic expression.
\begin{proposition}\label{prop-WStir}
   Let \(\phi\) be a Bernstein function. Then, as \(|z|=|a+ib|\to\infty\) with \(a=\Re(z)>0\) fixed, we have
   \begin{equation}\label{WStirAbs}
        \abs{\Wphi(a+ib)} \;=\; \lbrb{\sqrt{\phi(1)}\,e^{-T_\phi}+ \so{1}}e^{\int_1^a\ln\phi(u)du}
        \frac{e^{-A_\phi(a+ib)}}{\abs{\sqrt{\phi(a+1+ib)}}}.
    \end{equation}
\end{proposition}
\begin{remark}[\textit{\textbf{Further comments}}]\label{prop_WStirRF}
 Remark \ref{repWThF} contains discussion about the properties of \(A_\phi\) including expressions for its evaluation. However, when \(\phi\) is a Wiener-Hopf factor of a general \LL process, one can relate \(A_\phi\) to the original process, see Subsection \ref{subsec:BGWH}.
\end{remark}
Finally, we conclude this section by emphasising that equation~\eqref{receq} is not exclusive to the introduction of the \BG functions; see again~\cite{Web97} for more information. We have chosen to single out our case here because of the central role these functions play in probability theory.

Explicit examples of Bernstein-gamma functions from the work of Patie and Bartholmé~\cite{BarPat21}, as well as new ones, are provided in Appendix~\ref{sec: BG examples}.

\section{Exponential functionals 
}\label{secExpFunc}
If \(\xi\) is a \LLP with a \LLK exponent \(\Psi\), we formally introduce
\begin{equation}\label{EF}
    \IPsi := \IntOI e^{-\xi_s}\,\D s 
    \;=\; \int_{0}^{\mathbf{e}_{-\Psi(0)}} e^{-\xi_s}\,\D s,   
\end{equation}
where \(-\Psi(0)\geq 0\) is the killing rate of \(\xi\). From the law of large numbers and~\eqref{WHf} below, it follows directly that
\begin{equation}\label{ExistenceEF}
   \IPsi < \infty \;\text{ a.s.} \;\;\iff\;\; 
   \big(\Psi(0)<0\big) 
   \;\text{ or }\; 
   \big(\Psi(0)=0 \;\text{ and }\; \limi{t}\xi_t=\infty \;\text{ a.s.}\big) 
   \;\;\iff\;\; \phi_-(0) > 0.
\end{equation}
The fact that \(\phi_-(0)>0\) means that the descending ladder height process is killed.

\subsection{General expression of the Mellin transform and evaluation of the moments. Rate of decay of the Mellin transform. Specific examples}
\label{sec: Mellin transform and moments}
Recall that any \LLK exponent of a potentially killed \LLP \(\xi\) admits a Wiener-Hopf factorisation
\begin{equation}\label{WHf}
    \Psi(z) = -\phi_+(-z)\phi_-(z),\quad
    \text{at least for } z \in i\Rb,
\end{equation}
where
\begin{equation}\label{phipm}
  \phi_\pm(z) = \phi_\pm(0) + \dbf_\pm z + \IntOI \lbrb{1 - e^{-z y}} \mu_\pm(\D y),\quad \text{for } z \in \CbOOI,
\end{equation}
are two Bernstein functions associated with the bivariate ascending and descending ladder time and ladder height processes of \(\xi\). More precisely, if \(\tau^\pm\) are the inverse local times of the process \(\xi\) reflected at the supremum (for \(+\)) and at the infimum (for \(-\)), then \((\tau^\pm, \xi_{\tau_\pm})\) is a bivariate subordinator with Laplace exponent \(\kappa_\pm\); see~\cite[Chap.~VI]{Ber96} for more information. Then the functions \(\phi_\pm\) from~\eqref{phipm} are, up to the constant \(h(-\Psi(0))\), exactly \(\kappa_\pm(-\Psi(0),z)\); see~\eqref{WHf1} below.

We note that if the killing rate of \(\xi\) varies, i.e., \(-\Psi(0) = q \geq 0\), then the Wiener-Hopf factorisation is given by
\begin{equation}\label{WHf1}
 \Psi(z) = -h(q)\,\kappa_+(q, -z)\,\kappa_-(q, z),\quad 
 \text{for } q \in \lbbrb{0,\infty}, \quad 
 \text{at least for } z \in i\Rb,
\end{equation}
where, on \(\Cb_{\lbbrb{0,\infty}}\),
\begin{equation}\label{kappa}
 \kappa_{\pm}(q, z) 
   := -\log_0\Ebb{e^{-q\tau^\pm_1 - z\xi_{\tau^\pm_1}}}
   = c_\pm \exp\lbrb{\IntOI \int_{[0,\infty)} 
        \left(\frac{e^{-t} - e^{-q t - z x}}{t}\right) 
        \Pbb{\xi_t \in \pm \D y}\,\D t},
\end{equation}
and
\begin{equation}\label{h}
 h(q) := \exp\lbrb{-\IntOI 
        \left(\frac{e^{-t} - e^{-q t}}{t}\right) 
        \Pbb{\xi_t = 0}\,\D t}.
\end{equation}
The constants \(c_\pm\) can be chosen so that \(c_+ c_- = 1\). Relation~\eqref{WHf1} can be reconstructed from the classical books \cite{Ber96,Don07}, but it is presented in a neat form in~\cite[Sec.~2.2]{Minchev-Savov-2025}. Recalling~\eqref{aph}, note that for both \(\phi_\pm\) we set
\begin{equation}\label{aphiPsi}
     \ak_\pm := \ak_{\phi_\pm}, 
     \quad \uk_\pm := \uk_{\phi_\pm}, 
     \quad \text{and} \quad \bar{\ak}_\pm := \bar{\ak}_{\phi_\pm}.
\end{equation}
Using the fact that \(\phi_\pm\) are well-defined on \(\Cb_{\lbbrb{0,\infty}}\), equation~\eqref{WHf} yields immediately
\nomenclature[C]{$\ak_\pm^\Psi$, $\bar{a}_\pm^\Psi$,
$\uk_\pm^\Psi$}{analytical characteristics of the \LLK exponent $\Psi$, see \eqref{aPsi}}
\nomenclature[C]{$c^\Psi$}{$-\ak^\Psi_+ \ind{\uk^\Psi_+ = 0} \in[-\infty,0]$, Theorem \ref{thm: MPsi}}
\begin{equation}\label{aPsi}
\begin{alignedat}{3}
  \ak_+^\Psi &:= \sup\curly{u \ge 0:\,\Psi \in \Ac_{(0,u)}} = -\ak_+,
  &\quad 
  \uk_+^\Psi &:= -\uk_+,
  &\quad \bar{\ak}_+^\Psi &:= \min\curly{\ak_+^\Psi,\uk_+^\Psi} = -\bar{\ak}_+,\\
  \ak_-^\Psi &:= \inf\curly{u \le 0:\,\Psi \in \Ac_{(u,0)}} = \ak_-,
  &\quad 
  \uk_-^\Psi &:= \uk_-,
  &\quad \bar{\ak}_-^\Psi &:= \max\curly{\ak_-^\Psi,\uk_-^\Psi} = \bar{\ak}_-
\end{alignedat}
\end{equation}
because the regions of analyticity and the zeroes of \(\phi_\pm\) are transferred to those of \(\Psi\) and vice versa. However, when \eqref{ExistenceEF} holds, then necessarily
\begin{equation}\label{aPsiSpecial}
\uk_-^\Psi < 0, \quad \text{and therefore} \quad \bar{\ak}_-^\Psi = 0 \iff \ak_-^\Psi = 0,
\end{equation}
and also
\begin{equation}\label{uPsiSpecial}
    \uk_+^\Psi = 0 \quad \iff \quad \Psi(0) = 0
\end{equation}
or the \LLP pertaining to \(\Psi\) is not killed.
The last two relations are not mentioned in \cite{PatieSavov2018}, but they clarify the possible values of these crucial parameters.
Next, we define formally the Mellin transform of the law of \(\IPsi\) via
\begin{equation}\label{MT}
    \MPsi(z) := \Ebb{\IPsi^{z-1}} = \IntOI x^{z-1} \Pbb{\IPsi \in \D x}.
\end{equation}
Note that \(\MPsi(z)\) is always well-defined on \(\Cb_1\).
The next theorem collects the main known results concerning the complex-analytical properties of \(\MPsi\), and complements them with new existence results at the boundaries of the corresponding strips. 

\begin{theorem}[Thm.~2.1 from \cite{PatieSavov2018}]\label{thm: MPsi}
Let \(\Psi\) be the Lévy–Khintchine exponent of a potentially killed \LLP such that \(\phi_-(0)>0\).  
 Then
    \begin{equation}\label{MPsiAn}
        \MPsi(z) = \phi_-(0) \frac{\Gamma(z)}{\Wpp(z)} \Wpn(1-z) 
        \in \Att_{\lbrb{c^\Psi,\,1-\bar{\ak}^\Psi_-}} \cap \Mtt_{\lbrb{-\ak^\Psi_+,\,1-\bar{\ak}^\Psi_-}},
    \end{equation}
    where \(c^\Psi := -\ak^\Psi_+ \ind{\uk^\Psi_+ = 0} \leq 0\).  
    The following table describes precisely how \(\MPsi\) extends (or fails to) at the two boundaries of the analyticity strip \(\Cb_{(c^\Psi,\,1-\bar{\ak}^\Psi_-)}\):
\begin{center}
\begin{minipage}{\textwidth}
\captionof{table}{Boundary extensions of \(\MPsi(z)\) by parameter conditions}
\label{tab:psi-boundary}
\centering
\small
\renewcommand{\arraystretch}{1.25}
\setlength{\tabcolsep}{4pt}

\begin{tabular}{p{0.27\textwidth}p{0.34\textwidth}p{0.31\textwidth}}
\hline
\textbf{Parameter regime}
&
\textbf{Additional condition}
&
\textbf{Extension of \(\MPsi(z)\) at the boundary}
\\
\hline

\multicolumn{3}{l}{\textit{Right boundary}}\\
\softline

\(\bar{\ak}^{\Psi}_{-}=0\)
&
&
\(\MPsi\in\Att_{\lbrbb{c^{\Psi},\,1}}\)
\\
\softline

\multirow{2}{0.27\textwidth}{%
  \parbox{0.27\textwidth}{\raggedright
  \(\bar{\ak}^{\Psi}_{-}\in(-\infty,0)\)}%
}
&
\(\bar{\ak}^{\Psi}_{-}=\uk^{\Psi}_{-}<0\)
&
\(\MPsi\) does not extend continuously to
\(\Cb_{1-\bar{\ak}^{\Psi}_{-}}\)
\\
\softcline

&
\(\bar{\ak}^{\Psi}_{-}=\ak^{\Psi}_{-}>\uk^{\Psi}_{-}=-\infty\)
&
\(\MPsi\in\Att_{\lbrbb{c^{\Psi},\,1-\bar{\ak}^{\Psi}_{-}}}\)
\\
\softline

\(\bar{\ak}^{\Psi}_{-}=-\infty\)
&
\(\phi_{-}(z)\equiv\phi_{-}(\infty)\)
&
\(\MPsi\in\Att_{\lbrb{c^{\Psi},\,\infty}}\)
\\

\blockline
\multicolumn{3}{l}{\textit{Left boundary}}\\
\softline

\multirow{2}{0.27\textwidth}{%
  \parbox{0.27\textwidth}{\raggedright \(c^{\Psi}=0\)}%
}
&
\(\uk^{\Psi}_{+}=0\) and \(\phi_{+}'(0^+)<\infty\)
&
\(\MPsi\in\Att_{\lbbrb{0,\,1-\bar{\ak}^{\Psi}_{-}}}\)
\\
\softcline

&
\emph{in all remaining cases}
&
\(\MPsi\) extends continuously solely to
\(\Cb_{0}\setminus\{0\}\)
\\
\softline

\multirow{2}{0.27\textwidth}{%
  \parbox{0.27\textwidth}{\raggedright \(c^{\Psi}\in(-\infty,0)\)}%
}
&
\(\displaystyle \lim_{x\downarrow c^{\Psi}}|\phi_{+}(x)|<\infty\)
&
\(\MPsi\in\Att_{\lbbrb{c^{\Psi},\,1-\bar{\ak}^{\Psi}_{-}}}\)
\\
\softcline

&
\emph{in all remaining cases}
&
\(\MPsi\) does not extend continuously to \(\Cb_{c^{\Psi}}\)
\\
\softline

\(c^{\Psi}=-\ak^{\Psi}_{+}=-\infty\)
&
&
\(\MPsi\in\Att_{\lbrb{-\infty,\,1-\bar{\ak}^{\Psi}_{-}}}\)
\\

\hline
\end{tabular}
\end{minipage}
\end{center}
Finally, suppose \(\bar{\ak}^\Psi_+ > 0\). If \(\uk^\Psi_+ = \infty\) or \(\uk^\Psi_+ \notin \Nb\), then \(\MPsi\) has simple poles at all points \(-n\) such that \(0 \leq n < \ak^{\Psi}_+\). Otherwise, if \(\uk^\Psi_+ \in \Nb\), then \(\MPsi\) has simple poles at \(-n\) for all \(n \in \Nb \setminus \{\uk^\Psi_+, \uk^\Psi_+ + 1, \uk^\Psi_+ + 2, \ldots\}\). In both cases, the residues are given by \(\phi_+(0)\prod_{j=1}^n \Psi(j)/n!\).
\end{theorem}

\begin{remark}[\textbf{Historical remarks}]\label{MpsiH}
  Thanks to \eqref{receq} for \(\Gamma,\,\Wpp,\) and \(\Wpn\), expression~\eqref{MPsiAn} is in general a formal solution to the recurrence equation
  \begin{equation}\label{receqEx}
      \MPsi(z+1) = \frac{-z}{\Psi(-z)} \MPsi(z)
      \quad \text{on} \quad \curly{z \in i\Rb : z/\Psi(z)\ \text{is well-defined}}.
  \end{equation}
  It is well known that:
  \begin{itemize}
      \item \(\Psi(z)\) never vanishes on \(i\Rb\) when \(-\Psi(0)>0\);
      \item \(\Psi(z) \neq 0\) for \(z \in i\Rb \setminus \{0\}\) when \(\Pi\) is not supported on a lattice;
      \item if the support of \(\Pi\) coincides with \((h n)_{n \in \Zb}\) for some \(h > 0\) and \(\Psi(0)=0\), then \(\Psi(z) = 0\) if and only if \(z = 2\pi i k/h\) for \(k \in \Zb\).
  \end{itemize}
  Equation~\eqref{receqEx} has a long history. Its first appearance seems to be in \cite{CarPetYor97,CarPetYor01}, where the authors assume the existence of \(\Psi\) on \(\Rb\) and consider it on \(\Rb\), allowing for indeterminacy (i.e., both sides may be infinite). In \cite{Maulik-Zwart-06}, the recurrence is proved under weaker assumptions and solved under further conditions when \(\Psi\) pertains to a subordinator or a spectrally negative \LL process. Equation~\eqref{receqEx} has since been extensively used in settings where it is a priori defined on a strip and additional information on \(\Psi\) is available; see \cite{HackKuz14,Kuz2012,KuzPar13}. In some of these cases, one can obtain the law of \(\IPsi\); see Sections~\ref{hyperLP}--\ref{merLP}.
\end{remark}

\begin{remark}[\textbf{Further comments}]\label{MpsiF}
 Note that in \cite[Thm~2.1]{PatieSavov2018} it is stated (using our notation) that 
 \[
   c^\Psi = -\ak^\Psi_+ \ind{\bar{\ak}^\Psi_+ = 0},
 \]
 which coincides with our expression 
 \[
   c^\Psi = -\ak^\Psi_+ \ind{\uk^\Psi_+ = 0},
 \]
 since if \(\ak^\Psi_+ > 0\), then \(c^\Psi = 0 \iff \uk^{\Psi}_+ = \bar{\ak}^{\Psi}_+ = 0\). 
 When \(c^\Psi = 0\), it is claimed in \cite[Thm~2.1]{PatieSavov2018} that \(\MPsi \in \Att_{\lbbrb{0, 1-\bar{\ak}^\Psi_-}}\) if \(\phi'_+(0^+) < \infty\) and \(\bar{\ak}^{\Psi}_+ = 0\). 
 However, this is not correct, since the requirement \(\bar{\ak}^{\Psi}_+ = 0\) must in fact be \(\uk^\Psi_+ = 0\). 
 Indeed, if \(\phi_+(0) > 0\) and \(\ak^\Psi_+=\bar{\ak}^{\Psi}_+=0\), then the pole of \(\Gamma\) at zero is not cancelled, because \(\Wpp(0)\) is finite; see \cite[Proof of Thm.~2.1]{PatieSavov2018}.
\end{remark}
\begin{proof}
The only part that is not discussed explicitly in the literature concerns the extension (or non-extension) of \(\MPsi\) at the boundaries of the analyticity strip \(\Cb_{(c^\Psi,\,1-\bar{\ak}^\Psi_-)}\).

Suppose first that \(\bar{\ak}^\Psi_- = 0\). Then \(\MPsi(z)=\Eb[{\IPsi^{z-1}}]\) extends continuously to \(\Cb_{1}\). The existence of \(\MPsi\) at \(1-\bar{\ak}^\Psi_-\) (see \eqref{MPsiAn}) depends solely on the behaviour of \(\Wpn(1-z)\) at this point.

If \(\bar{\ak}^\Psi_- < 0\) and \(\bar{\ak}^\Psi_- = \uk^\Psi_- < 0\), then from the recurrence \eqref{receq} applied to \(\Wpn\), we obtain
\[
    \lim_{a \downarrow  \uk^\Psi_-} \Wpn(a) 
    = \lim_{a \downarrow  \uk^\Psi_-} \frac{\Wpn(a+1)}{\phi_-(\uk^\Psi_-)} 
    = \infty,
\]
since \(\Wpn(\uk^\Psi_- + 1)\) is non-zero (see Theorem~\ref{W-prop}). Therefore, in this case, \(\MPsi\) does \emph{not} extend continuously to \(\Cb_{1-\bar{\ak}^\Psi_-}\).

However, if \(\bar{\ak}^\Psi_- = \ak^{\Psi}_- > \uk^\Psi_- = -\infty\), then 
\[
    \Wpn(\ak^{\Psi}_-) = \frac{\Wpn(1+\ak^{\Psi}_-)}{\phi_-(\ak^{\Psi}_-)}
\]
is well-defined since \(\phi_-(\ak^{\Psi}_-) \in (0,\infty]\). Hence, in this case, \(\MPsi \in \Att_{\lbrbb{c^\Psi,\,1-\bar{\ak}^\Psi_-}}\). The case \(\bar{\ak}^\Psi_- = -\infty\) corresponds to \(\phi_-(z)\equiv \phi_-(\infty)>0\).

Next, consider the boundary at \(c^\Psi\). The existence of \(\MPsi\) at \(c^\Psi\) depends on the behaviour of \(\Gamma(z)/\Wpp(z)\) as \(z \to c^\Psi\). The case \(c^\Psi=0\) is discussed in Remark~\ref{MpsiF}. If \(c^\Psi < 0\), then \(\uk^\Psi_+ = 0\) and \(c^\Psi=-\ak^\Psi_+ < 0\), so
\[
    \abs{\phi_+(c^\Psi+)} := \lim_{x \downarrow c^\Psi} \abs{\phi_+(x)} \in (0,\infty].
\]
Then, using \eqref{receq}, we get
\[
    \lim_{a \downarrow c^\Psi} \frac{\Gamma(a)}{\Wpp(a)}
    = \lim_{a \downarrow c^\Psi} \frac{\Gamma(1+a)}{\Wpp(1+a)} \frac{\phi_+(a)}{a}
    = \frac{\phi_+(c^\Psi+)}{c^\Psi} \lim_{a \downarrow c^\Psi}\frac{\Gamma(1+a)}{\Wpp(1+a)}.
\]
If \(|\phi_+(c^\Psi+)| < \infty\), then
\begin{itemize}
    \item
if \(c^\Psi \notin -\Nb\), the last limit is well-defined, so \(\MPsi \in \Att_{\lbbrb{c^\Psi,\,1-\bar{\ak}^\Psi_-}}\);
\item if \(c^\Psi \in -\Nb\), then, since \(\uk^\Psi_+ = 0\) by the proof of \cite[Thm~2.1]{PatieSavov2018}, the term \(\Gamma(1+c^\Psi)/\Wpp(1+c^\Psi)\) is well-defined, so again \(\MPsi \in \Att_{\lbbrb{c^\Psi,\,1-\bar{\ak}^\Psi_-}}\).
\end{itemize}
If, however, \(\phi_+(c^\Psi+)=\infty\), then the limit above diverges to infinity. The case \(c^\Psi=-\infty\) can be handled by applying the recurrence relation as above.
\end{proof}
For convenience of the reader, we translate the results of Theorem~\ref{thm: MPsi} to the finiteness of the moments of~\(\IPsi\).

\begin{corollary}[Thm.~2.4 in \cite{PatieSavov2018}]\label{moments}
 Let \(\Psi\) be the \LLK exponent of a potentially killed \LL process such that \(\phi_-(0) > 0\). Then \[
\Ebb{\IPsi^{z}}
 \text{ is well-defined for }z \in \Cb_{\lbrb{-1 + c^\Psi, -\bar{\ak}^\Psi_-}},\] where \(c^\Psi = -\ak^\Psi_+ \ind{\uk^\Psi_+ = 0} \leq 0\). 
Moreover, at the two boundary points \(-1+c^\Psi\) and \( -\bar\ak_-^\Psi\), one has the following behaviour:
\pagebreak
\begin{center}
\captionof{table}{Boundary behaviour of $\Ebb{\IPsi^{z}}$ by parameter conditions}
\label{tab:psi-moment-boundary}

\small
\renewcommand{\arraystretch}{1.25}
\setlength{\tabcolsep}{4pt}

\begin{tabular}{p{0.30\textwidth}p{0.34\textwidth}p{0.28\textwidth}}
\hline
\textbf{Parameter regime}
&
\textbf{Additional condition}
&
\textbf{Behaviour of $\Ebb{\IPsi^{z}}$ at the boundary}
\\
\hline

\multicolumn{3}{l}{\textit{Right boundary}}\\
\softline

$\bar{\ak}^{\Psi}_{-}=0$
&
&
$\Ebbt{\IPsi^{z}}$ is well-defined for
$z\in\Cb_{\lbrbb{-1+c^{\Psi},\,0}}$
\\
\softline

\multirow{2}{0.30\textwidth}{%
  \parbox{0.30\textwidth}{\raggedright
  $\bar{\ak}^{\Psi}_{-}\in(-\infty,0)$}
}
&
$\bar{\ak}^{\Psi}_{-}=\uk^{\Psi}_{-}\in(-\infty,0)$
&
$\Ebbt{\IPsi^{-\bar{\ak}^{\Psi}_{-}}}=\infty$
\\
\softcline

&
$\bar{\ak}^{\Psi}_{-}=\ak^{\Psi}_{-}>\uk^{\Psi}_{-}=-\infty$
&
$\Ebbt{\IPsi^{-\bar{\ak}^{\Psi}_{-}}}$ is finite and given by
\eqref{MPsiAn}
\\

\blockline
\multicolumn{3}{l}{\textit{Left boundary}}\\
\softline

\multirow{2}{0.30\textwidth}{%
  \parbox{0.30\textwidth}{\raggedright $c^{\Psi}=0$}
}
&
$\uk^{\Psi}_{+}=0$ and $\phi_+'(0^+)<\infty$
&
$\Ebb{\IPsi^{-1}}=\Ebb{\xi_1}<\infty$
\\
\softcline

&
\emph{in all remaining cases}
&
$\Ebb{\IPsi^{-1}}=\infty$
\\
\softline

\multirow{2}{0.30\textwidth}{%
  \parbox{0.30\textwidth}{\raggedright
  $c^{\Psi}\in(-\infty,0)$}
}
&
$\displaystyle \lim_{x\downarrow c^{\Psi}}\abs{\phi_{+}(x)}<\infty$
&
$\Ebbt{\IPsi^{-1+c^{\Psi}}}$ is finite and given by
\eqref{MPsiAn}
\\
\softcline

&
\emph{in all remaining cases}
&
$\Ebbt{\IPsi^{-1+c^{\Psi}}}=\infty$
\\

\hline
\end{tabular}
\end{center}
\end{corollary}

\begin{remark}[\textbf{Historical remarks}]\label{momentsH}
    The existence of moments of \(\IPsi\) follows from the region of analyticity in Theorem~\ref{thm: MPsi}. All the papers referred to at the end of Remark~\ref{MpsiH} provide not only information about the finiteness of moments but also explicit expressions for them. The fact that \(\Ebb{\IPsi^a} < \infty\) for all \(a \in \lbrb{-1,0}\) is proved in \cite[Prop.~2.6]{ArRiv23} and is equivalent to \(\MPsi \in \Att_{\lbrb{0,1}}\); see \eqref{MPsiAn}. The equivalence \(\Eb[{\IPsi^{-1}}] < \infty \iff (q = 0\) and \(\Ebb{\xi_1} < \infty\)) is also contained in \cite[Prop.~2.6]{ArRiv23} and follows from Theorem~\ref{thm: MPsi} when \(\MPsi\) extends to \(\Att_{\lbbrb{0,1-\bar{\ak}^\Psi_-}}\).
\end{remark}

\begin{remark}[\textbf{Further comments}]\label{momentsF}
    If \(\Psi(0) = 0\), i.e., if \(\xi\) is conservative, then necessarily \(\uk^\Psi_+ = \bar{\ak}^\Psi_+ = 0\) and \(c^\Psi=-\ak^\Psi_+\), since in this case \(\xi\) drifts to \(\infty\).
\end{remark}

\begin{remark}[\textbf{Conjecture of Bertoin and Yor}]\label{BYConj}
    The case \(c^\Psi = -\infty\) shows that for an unkilled \LLP with all positive exponential moments, \(\IPsi\) has all negative moments. It is a conjecture of Bertoin and Yor, see \cite[Sec.~3]{BerYor05}, that in this situation \(\IPsi^{-1}\) is moment determinate if and only if \(\Pi((0,\infty)) = 0\). The negative moments are given by
    \[
        \Ebb{\IPsi^{-n}} = \Ebb{\xi_1} \frac{ \prod_{j=1}^{n-1} \Psi(j) }{ (n-1)! },
    \]
    see \cite[Thm.~3]{BerYor05}, or equivalently \(\Ebb{\IPsi^{-n}} = \MPsi(-n+1)\). If \(\Pi\lbrb{0,\infty} > 0\) and, for some \(a > 0\), either \(a\) is an atom of \(\Pi\) or \(\Pi((a,\infty)) > 0\), then one obtains
    \[
        \liminf_{j \to \infty} \Psi(j)e^{-a j} > 0,
    \]
    see \eqref{LLK}. Therefore, for large \(n\), \(\prod_{j=1}^{n-1} \Psi(j) > C e^{n(n-1)/2}\) for some constant \(C > 0\), and the moment conditions for moment determinacy of \(\IPsi^{-1}\) are not satisfied; see \cite[p.~4]{Lin17}. We return to this discussion in the section where we present results on \(\pPs\), since conditions for moment indeterminacy invariably include information about the density.
\end{remark}

\begin{proof}
    The only cases not explicitly addressed in the literature are items~2 and~4. These follow immediately by appeal to Theorem \ref{moments} and  substitution into the identity \(\MPsi(z+1) = \Ebb{\IPsi^z}\) at \(z = -1 + c^\Psi\) and \(z = -\bar{\ak}^\Psi_-\), respectively.
\end{proof}

For completeness, we spell out the trivial evaluation of integer moments of \(\IPsi\), whenever these exist. We begin with the positive integer moments.  
\begin{corollary}[\cite{CarPetYor94}; Thm.~2.4 in \cite{PatieSavov2018}]\label{cor:PosMoments}
    Let \(\Psi\) be the \LLK exponent of a potentially killed \LL process such that \(\phi_-(0) > 0\).  
    If \(-\bar{\ak}^\Psi_- \notin \Nb\) and \(-\bar{\ak}^\Psi_- > 1\), then for any integer \(n\) with \(1 \leq n < -\bar{\ak}^\Psi_-\), we have
    \begin{equation}\label{cor:PosMoments_1}
        \Ebb{\IPsi^n}
        = (-1)^n \,\frac{n!}{\prod_{j=1}^n \Psi(-j)}.
    \end{equation}
    Furthermore, if \(1 \leq -\bar{\ak}^\Psi_- \in \Nb\) and \(\bar{\ak}^\Psi_- = \ak^\Psi_- > \uk^\Psi_- = -\infty\), then \eqref{cor:PosMoments_1} also holds for \( n  = -\bar{\ak}^\Psi_-\).
\end{corollary}

\begin{remark}[\textbf{Historical remarks}]\label{positivemomentsF}
   The first results on positive integer moments appear to be contained in \cite{CarPetYor94}, with subsequent works by various authors providing additional refinements and extensions.
\end{remark}

We proceed with the negative integer moments.
\begin{corollary}[Thm.~3 in \cite{BerYor05}]\label{cor:PosMoments1}
    Let \(\Psi\) be the \LLK exponent of a potentially killed \LL process such that \(\phi_-(0) > 0\), \(\Psi(0)=0\), and \(\Psi'(0)=\Ebb{\xi_1}<\infty\).  
    If \(-1 + c^{\Psi} = -1 - \ak^\Psi_+ < -1\) and \(1 - c^{\Psi} \notin \Nb\), then, for any integer \(1 \leq n < 1 - c^{\Psi}\),
    \begin{equation}\label{cor:PosMoments1_1}
        \Ebb{\IPsi^{-n}}
        = \Ebb{\xi_1}\,\frac{ \prod_{j=1}^{n-1} \Psi(j) }{ (n-1)! },
    \end{equation}
    with the convention that \(\Ebb{\IPsi^{-1}} = \Ebb{\xi_1}\).  
    Furthermore, if \(1 \leq 1 - c^{\Psi} \in \Nb\) and \(\lim_{x \downarrow c^{\Psi}} \abs{\phi_+(x)} < \infty\) or equivalently \(\lim_{x\uparrow -\ak^\Psi_+}\abs{\Psi(x)}<\infty\), then \eqref{cor:PosMoments1_1} also holds for \(n = 1 - c^{\Psi}\).
\end{corollary}

Next, we discuss the rate of decay of \(\MPsi\) along complex lines, a feature that is particularly useful for studying the law of \(\IPsi\) via Mellin inversion. We begin with some notation. Recall that \(\dbf_\pm\) denote the drift terms of the Wiener–Hopf factors \(\phi_\pm\); see \eqref{phipm}. We also set \(\overline{\Pi}(0) := \Pi\lbrb{\Rb}\) and recall that \(\overline{\mu}_\pm(0) := \mu_\pm\lbrb{\Rp}\), where \(\mu_\pm\) are the Lévy measures of \(\phi_\pm\) (again see \eqref{phipm}).  

If \(\overline{\Pi}(0) < \infty\) and \(\dbf_+ > 0\), a result of Vigon \cite{Vigon02}, stated in full generality in \cite[Prop.~B.2]{PatieSavov2018}, yields that 
\[
\mu_-(\D y) = v_-(y)\,\D y, \quad y > 0, \qquad
v_- \text{ is right-continuous on } \intervalOOI, \quad
\text{and } v_-(0^+) \in \intervalOOI.
\]

Finally, we introduce a quantity that describes the smoothness of the law of \(\IPsi\) and the polynomial decay of \(\MPsi\)
\begin{equation}\label{Npsi}
    \NPs := 
\begin{cases}
  \dfrac{ v_-(0^+)}{\phi_-(0) + \overline{\mu}_-(0)} + \dfrac{\phi_+(0) + \overline{\mu}_+(0)}{\dbf_+}, 
   & \text{if } \dbf_+ > 0,\, \dbf_- = 0, \text{ and } \overline{\Pi}(0) < \infty, \\[1.2ex]
  \infty, & \text{otherwise.}
\end{cases}
\end{equation}

We are now in a position to state the main result concerning the decay of \(\MPsi\) along complex lines.
\begin{theorem}[Thm.~2.3 in \cite{PatieSavov2018}]\label{MDecay}
    Let \(\Psi\) be the \LLK exponent of a potentially killed \LLP such that \(\phi_-(0) > 0\).  
    \begin{itemize}
        \item If \(\NPs = \infty\), then for all \(\beta > 0\) and any \(a \in (0,\, 1 - \bar{a}_-^\Psi)\),
        \begin{equation}\label{MPsiDecay}
            \limsupi{|b|} |b|^\beta \abs{\MPsi(a + ib)} = 0.
        \end{equation}
        \item If instead \(\NPs \in \lbrb{0, \infty}\), then for any \(a \in \lbrb{0,\, 1 - \bar{a}_-^\Psi}\),
        \begin{equation}\label{MPsiDecay3.11}
            \limsupi{|b|} |b|^\beta \abs{\MPsi(a + ib)} =
            \begin{cases}
               0, & \text{if } \beta < \NPs, \\[4pt]
               \infty, & \text{if } \beta > \NPs.
            \end{cases}
        \end{equation}
        \item Finally, \(\NPs = 0\) if and only if \(\xi_t = \dbf_+ t\) for all \(t \geq 0\).
    \end{itemize}
\end{theorem}
\begin{remark}[\textbf{Historical remarks}]\label{MDecayH}
When solving \eqref{receqEx} for specific cases, such as those in subsections~\ref{hyperLP} and~\ref{merLP}, one uses a potential candidate for \(\MPsi\) and verifies that it is indeed \(\MPsi\). This is accomplished by analysing the rate of decay of the candidate for \(\MPsi\) along complex lines and the properties of entire periodic functions; see \cite{HackKuz14, Kuz2012, KuzPar13}. A somewhat general investigation of the rate of decay seems to be first conducted in \cite[Thm~6.0.2]{PatSav21}. Full study is offered in \cite{PatieSavov2018} and as in the specific cases mentioned above served as the tool to identify the Mellin transform of \(\IPsi\) in complete generality.
\end{remark}
\begin{remark}[\textbf{Further comments}]\label{MDecayF}
    Note that \(\NPs < \infty\) if and only if \(\xi\) is a compound Poisson process with positive drift, that is, \(\dbf_+ > 0,\; \dbf_- = 0,\) and \(\overline{\Pi}(0) < \infty\). In all other cases, the decay of \(\lvert \MPsi(a+ib) \rvert\) is faster than any polynomial, and in many situations even exponential.
\end{remark}
The next corollary shows that the rate of decay established in Theorem~\ref{MDecay} is preserved on the strip of meromorphicity \(\Cb_{\lbrb{-\ak^\Psi_+,\,1-\bar{\ak}^\Psi_-}}\); see Theorem~\ref{thm: MPsi}.

\begin{corollary}[Prop.~7.1 in \cite{PatieSavov2018}]\label{corMDecay}
    Let \(\Psi\) be a \LLK exponent of a potentially killed \LLP such that \(\phi_-(0) > 0\). Then, for any \(a \in (-\ak^\Psi_+, 0]\), relations~\eqref{MPsiDecay} and~\eqref{MPsiDecay3.11} hold with the same value of \(\NPs\). Moreover, if \(\lvert \phi_+(-\ak^{\Psi}_+) \rvert < \infty\), which is always the case when \(\ak^\Psi_+ = 0\), then~\eqref{MPsiDecay} and~\eqref{MPsiDecay3.11} also hold for \(a = -\ak^\Psi_+\) with the same \(\NPs\).
\end{corollary}
\begin{remark}[\textbf{Further comments}]\label{corMDecayF}
    Let \(-\ak^\Psi_- < \uk^{\Psi}_- = \bar{\ak}^{\Psi}_- < 0\). Then the decay in~\eqref{MPsiDecay} and~\eqref{MPsiDecay3.11} is always preserved along \(\Cb_{1-\uk^{\Psi}_-}\) when \(\NPs < \infty\). Moreover, if the \LLP is \textit{weak non-lattice}, the decay is preserved also in the case \(\NPs = \infty\); see subsection~\ref{Lasymp}. All these statements follow from \cite[Prop.~7.2]{PatieSavov2018}.
\end{remark}
\begin{proof}
   The first claim is exactly the statement of \cite[Prop.~7.1]{PatieSavov2018}.  
   Suppose now that \(|{\phi_+(-\ak^{\Psi}_+)}| < \infty\). Then the second assertion follows from the  following identity in itself coming from \eqref{receqEx}
   \[
   \MPsi\left(-\ak^\Psi_+ + ib\right) 
   = \frac{\phi_+(-\ak^\Psi_+ + ib)\, \phi_-(\ak^\Psi_+ - ib)}{-\ak^\Psi_+ + ib}\, 
     \MPsi\left(1 - \ak^\Psi_+ + ib\right),
   \]
   see \cite[(7.11)]{PatieSavov2018}.  

   Indeed, if \(\NPs = \infty\), then under the assumption \(|{\phi_+(-\ak^{\Psi}_+)}| < \infty\), and using \cite[Prop.~3.1(4)]{PatieSavov2018}, we deduce
   \[
   \left| \frac{\phi_+(-\ak^\Psi_+ + ib)\, \phi_-(\ak^\Psi_+ - ib)}{-\ak^\Psi_+ + ib} \right| = \bo{|b|},
   \]
   so the super-polynomial decay of \(\MPsi\) is preserved along \(\Cb_{-\ak^\Psi_+}\).

   If instead \(\NPs < \infty\), then since \(\dbf_+ > 0\), the process \(\xi\) is non-lattice and therefore from \cite[Prop.~3.1(4-5)]{PatieSavov2018}
   \[
   \lim_{|b|\to\infty} \phi_-(\ak^\Psi_+  - ib) = \phi_-(\infty) < \infty,
   \quad \text{and} \quad
   \lim_{|b|\to\infty} \frac{\phi_+(-\ak^\Psi_+ + ib)}{-\ak^\Psi_+ +  ib} = \dbf_+ > 0.
   \]
   These limits guarantee that the boundary growth remains controlled, and thus relations~\eqref{MPsiDecay} and~\eqref{MPsiDecay3.11} hold throughout the full strip \(a \in [-\ak^\Psi_+, 0]\) with the same~\(\NPs\).
\end{proof}

The next result does not appear in the literature. It provides potentially useful uniform estimates on the rate of decay of~\(\MPsi\). We will use for positive functions \(f\) and \(g\) 
\[
   f \lesssim g \text{ to mean that } \limsup_{x\to \infty} \frac{f(x)}{g(x)} < \infty
\]
and
\[
    f \asymp g \quad \text{to mean that} \quad 
    0 < \liminf_{x\to\infty} \frac{f(x)}{g(x)} 
    \leq \limsup_{x\to\infty} \frac{f(x)}{g(x)} < \infty.
\]

\begin{proposition}\label{uniMpsi}
Let \(\Psi\) be the \LLK exponent of a potentially killed \LLP such that \(\phi_-(0) > 0\). Then, for any interval \((c,d) \subset (-\ak^\Psi_+,\, 1 - \bar{\ak}^\Psi_-)\), there exists a constant
$C:=C(c,d)>0$ such that, as $|b|\to\infty$,
\begin{equation}\label{uniMpsiEq}
   \sup_{a,a' \in (c,d)}
   \left| \frac{\MPsi(a+ib)}{\MPsi(a'+ib)} \right|
   \quad
   \begin{cases}
      \lesssim |b|^{C}, & \text{if } \NPs = \infty, \\[1ex]
      \asymp 1, & \text{if } \NPs < \infty.
   \end{cases}
\end{equation}
Moreover:
\begin{enumerate}
   \item if \(\bar{\ak}^\Psi_- = 0\), then~\eqref{uniMpsiEq} also holds for \((c,d] \subseteq (-\ak^\Psi_+,\, 1]\);
   \item if \(\lvert \phi_+(-\ak^\Psi_+) \rvert < \infty\) (which is the case when \(\ak^\Psi_+ = 0\)), then~\eqref{uniMpsiEq} holds for \([c,d) \subseteq [-\ak^\Psi_+,\, 1 - \bar{\ak}^\Psi_-)\);
   \item if both of the above conditions hold, then~\eqref{uniMpsiEq} holds for \([c,d] \subseteq [-\ak^\Psi_+,\, 1]\).
\end{enumerate}
\end{proposition}
\begin{remark}[\textbf{Further comments}]\label{uniMpsiF}
    Essentially,~\eqref{uniMpsiEq} shows that the decay of \(\abs{\MPsi(a+ib)}\) as \(|b| \to \infty\) is uniform over compact subsets of the real part \(a\). This property allows for the Mellin inversion to be taken at the boundary of the region of analyticity of the Mellin transform.
\end{remark}
\begin{proof}[Proof of Proposition~\ref{uniMpsi}]
From Proposition~\ref{UnifBGF} applied to \(\Gamma\), \(\Wpn\), and \(\Wpp\) in the representation of \(\MPsi\) (see~\eqref{MPsiAn}), we have
    \begin{equation*}
     \begin{split}
           &\sup_{a,a' \in (c,d)} \left| \frac{\MPsi(a+ib)}{\MPsi(a'+ib)} \right| \asymp\\
           &\qquad \sup_{a,a' \in (c,d)}\exp\left( \int_{1+a}^{1+a'} \ln|u+ib| - \ln|\phi_+(u+ib)| + \ln|\phi_-(1-u-ib)| \, \D u\right).
     \end{split}
    \end{equation*}
    Since, for any Bernstein function, \(\abs{\phi(z)}\lesssim |z|\) uniformly on
    \(\Cb_{(\baph,\infty)}\), and
    \(\abs{\phi(z)}=\dbf |z|(1+\so{1})\) when \(\dbf>0\), including the
    boundary line if \(|\phi(\baph)|<\infty\), see
    \cite[Prop.~3.1(4)]{PatieSavov2018}, the integral in the preceding display is
    bounded above by a constant multiple of \(\ln |b|\), where the constant may
    depend on \(c\) and \(d\), when \(\NPs=\infty\). Exponentiating gives the
    claimed polynomial bound in~\eqref{uniMpsiEq}, in all cases stated there.
    
    If \(\NPs < \infty\), then \(\dbf_+ > 0\) and \(\phi_-(\infty) < \infty\), so from \cite[Prop.~3.1(4-5)]{PatieSavov2018}
    \[
        \ln|u+ib| - \ln|\phi_+(u+ib)| = \bo{1}, \quad 
        \text{and}
        \quad
        \ln|\phi_-(1-u-ib)| = \bo{1}.
    \]
    Hence the integral remains uniformly bounded, and we conclude the desired estimate in all cases. This completes the proof of Proposition~\ref{uniMpsi}.
\end{proof}

The understanding of the super-polynomial decay exhibited by \(\abs{\MPsi}\) when \(\NPs = \infty\) can be further refined in some cases. The next theorem discusses exponential decrease.
\begin{theorem}[Thm.~2.3 in \cite{PatieSavov2018}]\label{MDecay1}
   Let \(\Psi\) be a \LLK exponent of a potentially killed \LLP such that \(\phi_-(0) > 0\). 
   If \(\dbf_- > 0\), i.e., \(\xi\) creeps downwards, or if \(\phi_+ = c \phi_-\), i.e., \(\xi\) is a symmetric \LL process, then, for any \(a \in (0,\, 1 - \bar{\ak}_-^\Psi)\),
   \begin{equation}\label{MPsiDecay1}
     \limsupi{|b|} \frac{\ln\abs{\MPsi(a+ib)}}{|b|} \leq -\frac{\pi}{2}.
   \end{equation}
   Moreover, if 
   \[
     \Theta_- := \liminfi{|b|} \frac{\int_0^{|b|} \arg\phi_-(1+iu)\, \D u}{|b|}, 
     \qquad
     \Theta_+ := \limsupi{|b|} \frac{\int_0^{|b|} \arg\phi_+(1+iu)\, \D u}{|b|},
   \]
   then \(\Theta_+, \Theta_- \in \lbbrbb{0,\, \pi/2}\), and for any \(a \in (0,\, 1 - \bar{\ak}_-^\Psi)\),
   \begin{equation}\label{MPsiDecay2}
     \limsupi{|b|} \frac{\ln\abs{\MPsi(a+ib)}}{|b|} \leq -\frac{\pi}{2} - \Theta_- + \Theta_+.
   \end{equation}
\end{theorem}
\begin{remark}[\textbf{Further comments}]\label{MDecay1F}
    Additional cases of decay are presented in \cite[Thm.~2.3]{PatieSavov2018}. These address situations where either of \(\phi_\pm\) is regularly varying at zero, but they require further assumptions on \(\mu_\pm\). For this reason, we do not apply them here.
\end{remark}

Next, we provide some examples where \(\MPsi\) can be computed in terms of well-known special functions.

\subsubsection{Brownian motion with drift}
If \(\xi\) is a potentially killed Brownian motion with killing rate \(q\), drift \(\mu\), and variance \(\sigma^2\), then 
\[
\Psi(z) = \frac{\sigma^2}{2} z^2 + \mu z - q 
= -\frac{\sigma^2}{2} (z + b_+) (-z - b_-),
\]
where 
\[
b_\pm := \frac{\mu}{\sigma^2} \pm \frac{1}{\sigma^2} \sqrt{\mu^2 + 2q\sigma^2}.
\]
In the Wiener–Hopf factorisation \eqref{WHf}, this corresponds to 
\[
\phi_+(z) = z - b_-, 
\qquad 
\phi_-(z) = \frac{\sigma^2}{2}(z + b_+).
\]

It follows that \(\IPsi < \infty\) almost surely if and only if \(q > 0\) or \(\mu > 0\); see \eqref{Existence}. From the first example in Appendix~\ref{sec: BG examples} and the last claim of Theorem~\ref{W-prop}, we obtain
\[
\Wpp(z) = \frac{\Gamma(z - b_-)}{\Gamma(1 - b_-)}, 
\qquad 
\Wpn(z) = \frac{\sigma^{2(z-1)}}{2^{z-1}} \frac{\Gamma(z + b_+)}{\Gamma(1 + b_+)}.
\]
Therefore, from \eqref{MPsiAn},
\begin{equation}\label{MPsiBM}
    \MPsi(z) 
    = b_+ \left(\frac{\sigma^{2-2z}}{2^{1-z}}\right) 
      \frac{\Gamma(z)\, \Gamma(1 - b_-)}{\Gamma(z - b_-)} 
      \frac{\Gamma(1 - z + b_+)}{\Gamma(1 + b_+)}.
\end{equation}

In the special case \(\mu=0\), \(\sigma^2=2\), and \(q>0\), corresponding to a killed symmetric Brownian motion, this reduces to
\[
    \MPsi(z) = \sqrt{q}\,\frac{\Gamma(z)}{\Gamma(z + \sqrt{q})}\,\Gamma(1 - z + \sqrt{q}).
\]
\subsubsection{Hypergeometric \LLPs}\label{hyperLP}
  Let \(\xi\) be a stable \LL\ process, i.e., for any \(c > 0\) it holds that 
\[
\lbrb{\xi_{ct}}_{t \geq 0} \stackrel{w}{=} c^{1/\alpha} \lbrb{\xi_t}_{t \geq 0},
\]
where \(\alpha \in \lbrbb{0,2}\), with \(\alpha = 2\) corresponding to the Brownian motion case. In \cite{CabCha06}, the following three processes derived from \(\xi\) are considered:
\begin{itemize}
  \item \(\xi\) started from \(x > 0\) and killed upon exiting the upper half-plane; 
  \item \(\xi\) started from zero and conditioned to stay positive; 
  \item \(\xi\) started from \(x > 0\) and conditioned to hit zero continuously. 
\end{itemize}
    These latter processes are positive self-similar Markov processes and thus can be represented as suitable random time changes of, respectively, three \LL processes called Lamperti stable \LL processes, say \(\xi^*,\, \xi^{\uparrow}\), and \( \xi^{\downarrow}\). This is a consequence of the Lamperti representation; see \cite{BerSav11, CabCha06, Lam72} for more information. For each of \(\xi^*,\, \xi^{\uparrow}\), and \( \xi^{\downarrow}\), one can define a respective exponential functional, which represents the absorption time of the corresponding positive self-similar Markov process or other quantity of interest. The processes \(\xi^*,\, \xi^{\uparrow}\), and \( \xi^{\downarrow}\) fall into the domain of the so-called hypergeometric \LL processes, see \cite{KuzPar13}. 

Hypergeometric \LLPs are defined via their \LLK exponent, which takes the form
\begin{equation}\label{hyper}
\begin{split}
   \Psi(z) &= -\frac{\Gamma\lbrb{1 - \beta + \gamma - z}}{\Gamma\lbrb{1 - \beta - z}}
             \frac{\Gamma\lbrb{\hat{\beta} + \hat{\gamma} + z}}{\Gamma\lbrb{\hat{\beta} + z}}, 
             \quad
             \text{for}
             \quad
 \beta \leq 1,\, \hat{\beta} \geq 0,\, 
 \text{and } \gamma, \hat{\gamma} \in \lbrb{0,1}.
\end{split}
\end{equation}
In \cite[Thm~2]{KuzPar13}, it is shown that if \(\xi\) is hypergeometric, then
\begin{equation}\label{MPsiHyper}
  \MPsi(z) = C\, \Gamma(z)\, \frac{G(1-\beta+z, 1)}{G(1-\beta+\gamma+z, 1)}\, \frac{G(\hat\beta+\hat\gamma+1-z, 1)}{G(\hat\beta+1-z, 1)},
\end{equation}
where \(G\) denotes the Barnes gamma function and \(C\) is a normalising constant such that \(\MPsi(1) = 1\). This result generalises the result on exponential functionals of hyper-exponential \LLPs presented in \cite{CaiKou12}.

\subsubsection{Meromorphic \LLPs and \LLPs with jumps of rational transform}\label{merLP}

Meromorphic \LL processes, as introduced in \cite{KuKyPa12}, are a versatile class of \LLPs containing many known classes of \LL processes, and further extending the aforementioned hypergeometric \LL processes; see \cite[Sec.~3]{KuKyPa12}. We recall that \(\xi\) is a meromorphic \LLP if its \LL measure has a density of the form
\[
\pi(x) = \ind{x>0} \sum_{n\geq 1} a_n \rho_n e^{-\rho_n x} + \ind{x<0} \sum_{n\geq 1} \hat{a}_n \hat\rho_n e^{\hat\rho_n x},
\]
where \(a_n, \rho_n, \hat{a}_n, \hat\rho_n\) are strictly increasing positive sequences tending to infinity and
\[
\sum_{n\geq 1} a_n \rho_n^{-2} < \infty \quad \text{and} \quad \sum_{n\geq 1} \hat{a}_n \hat\rho_n^{-2} < \infty,
\]
see \cite[Thm~1]{KuKyPa12} for the full characterisation of meromorphic \LL processes. In particular, \cite[Thm~1.]{KuKyPa12} implies that \(\Psi(z)\) extends to a real meromorphic function, and for any \(q > 0\), the equation \(\Psi(z) = q\) has only simple real solutions \(\chi_n(q)\) and \(-\hat\chi_n(q)\), \(n \geq 1\), such that, with \(\hat\rho_0 = 0\), one has
\[
\dots < -\hat\rho_2 < -\hat\chi_2(q) < -\hat\rho_1 < -\hat\chi_1(q) < 0 < \chi_1(q) < \rho_1 < \chi_2(q) < \rho_2 < \dots,
\]
see \cite{HackKuz14}. Then \cite[(3.6)]{HackKuz14} gives the following expression for the Mellin transform of the exponential functional of meromorphic \LL processes:
\begin{equation}\label{MerExp}
\begin{split}
     \MPsi(z) = C^{z-1}(q) \prod_{n=1}^\infty &
     \frac{\Gamma(\hat{\chi}_n(q) + 1)}{\Gamma(\hat{\rho}_{n-1} + 1)}
     \frac{\Gamma(\hat{\rho}_{n-1} + z)}{\Gamma(\hat{\chi}_n(q) + z)}
     \frac{(\hat{\chi}_n(q) + 1)^{z-1}}{(\hat{\rho}_{n-1} + 1)^{z-1}} \\
     &\hspace{2cm}  \times\frac{\Gamma(\rho_n)\, \Gamma(\chi_n(q) + 1 - z)}{\Gamma(\chi_n(q))\, \Gamma(\rho_n + 1 - z)}
     \frac{\chi_n(q)^{z-1}}{\rho_n^{z-1}},
\end{split}
\end{equation}
where
\[
C(q) := \frac{1}{q} \prod_{n=1}^\infty \frac{1 + \frac{1}{\hat\rho_n}}{1 + \frac{1}{\hat\chi_n(q)}}.
\]

\LLPs with jumps of rational transform are an extension of the hyper-exponential \LL processes, which are themselves contained in the class of meromorphic \LL processes. Their \LL measure has a density of the type
\[
\pi(x) = \ind{x>0} \sum_{j=1}^J \sum_{i=1}^{m_j} \alpha_{ij} x^{i-1} e^{-\rho_j x}
+ \ind{x<0} \sum_{j=1}^{\hat{J}} \sum_{i=1}^{\hat{m}_j} \hat\alpha_{ij} x^{i-1} e^{\hat\rho_j x},
\]
where \(m_j, \hat{m}_j \in \Nb\), \(\rho_j > 0, \hat\rho_j > 0\), and \(\rho_i \neq \rho_j, \hat\rho_i \neq \hat\rho_j\) for \(i \neq j\). Then, \cite[Thm~1]{Kuz2012} provides the Mellin transform of exponential functionals of \LLPs with jumps of rational transform as
\begin{equation*}
    \begin{split}
        \MPsi(z) &= A^{1-z} \Gamma(z) \frac{\Gc(z)}{\Gc(1)},
    \end{split}
\end{equation*}
where \(\Gc(z)\) is a ratio of Gamma functions depending on \(\rho_n, \hat{\rho}_n\) and the solutions \(\chi_n(q), \hat\chi_n(q)\) to \(\Psi(z) = q\), and \(A\) depends on the Brownian component \(\sigma^2\), the mean \(\mu\), the killing rate \(q\), and the rate of the Poisson counting process, as the jumps in this case are of finite activity.

\subsection{Law}
Let \(\PIp(\D x) := \Pbb{\IPsi \in \D x}\) denote the law of \(\IPsi\), and \(\PIp(x) := \Pbb{\IPsi \leq x}\) for \(x \geq 0\). We write \(\pPs\) for the density, which, from Theorem \ref{Existence}, is known to exist away from the pure drift case. In this section, we consider general results concerning the law of \(\IPsi\) and provide examples when the law can be computed as an infinite series.

\subsubsection{Absolute continuity, smoothness and analyticity}\label{SmoothSec}
We start with the existence of the density of \(\IPsi\). Since the relevant results are scattered in the literature, we collect them here.

\begin{theorem}[Thm.~3.9 in \cite{BeLiMa08}; Thm.~2.1 in \cite{PaRiSch13}; Thm.~2.4(1) in \cite{PatieSavov2012}]\label{Existence}
    Let \(\Psi\) be the \LLK exponent of a potentially killed \LLP such that \(\phi_-(0) > 0\). Then \(\Pbb{\IPsi \in \D x}\) admits a density \(\pPs\) with respect to the Lebesgue measure, except in the case where \(\xi\) is an unkilled linear drift, i.e., \(\xi_t = \dbf_+ t\) for \(t \geq 0\).
\end{theorem}

\begin{proof}
    The case \(-\Psi(0) > 0\) is covered by \cite[Thm.~2.1]{PaRiSch13}. When \(\Psi(0) = 0\), we invoke \cite[Thm.~3.9]{BeLiMa08}: if \(\Pi\) is not identically zero (corresponding, in the setting of \cite{BeLiMa08}, to \(g(t) = e^{-t}\)), then the conditions of \cite[Thm.~3.9(a)(i)]{BeLiMa08} apply and \(\IPsi\) possesses a density. If instead \(\Pi \equiv 0\) but \(\sigma^2 > 0\), the claim follows from \cite[Thm.~3.9(b)]{BeLiMa08}.
\end{proof}

We continue with the smoothness of the law. The space \(\mathtt{C}^{n}_0\lbrb{\Rb^+},\ n \geq 0\), denotes the set of continuous functions on \(\Rp\) possessing \(n\) continuous derivatives on \(\Rp\) which, together with the function itself, vanish at infinity. For \(n = \infty\), we write \(\mathtt{C}^{\infty}_0\lbrb{\Rb^+}\), and for \(n = 0\), we set \( \ccoRp:=\mathtt{C}^{0}_0\lbrb{\Rb^+} \). Let \(\lceil\cdot\rceil\) denote the ceiling function, that is,
\[
\lceil x \rceil := \min\curly{n \geq 0 : n \geq x}, \quad x \geq 0.
\]
Recall the definition of \(\NPs\); see \eqref{Npsi}.

\begin{theorem}[Thm~2.4 in \cite{PatieSavov2018}]\label{Smooth}
Let \(\Psi\) be a \LLK exponent of a potentially killed \LLP such that \(\phi_-(0) > 0\). Then \(\PIp \in \mathtt{C}^{\lceil \NPs \rceil - 1}_0\lbrb{\Rb^+}\). Moreover, if \(\NPs > 1\), then for any \(0 \leq n \leq \lceil \NPs \rceil - 2\) and any \(a \in (-\ak^\Psi_+ \ind{\uk^\Psi_+ = 0}, 1 - \bar{\ak}_-^\Psi)\),
\begin{equation}\label{Smoothf}
   \pPs^{(n)}(x) = \frac{(-1)^n \phi_-(0)}{2\pi i} \int_{a-i\infty}^{a+i\infty} x^{-z-n} \Gamma(z+n) \frac{\Wpn(1-z)}{\Wpp(z)} \D z,
\end{equation}
where the integral is absolutely convergent for \(x > 0\). Consequently, if \(\xi\) is not a compound Poisson process with positive drift, i.e., \(\dbf_+ > 0,\ \dbf_- = 0,\) and \(\overline{\Pi}(0) < \infty\), then \(\pPs \in \mathtt{C}^{\infty}_0\lbrb{\Rb^+}\). Finally, if \(0 < \NPs < 1\), then for any \(0 < \alpha < \NPs\), \(\PIp\) is Hölder continuous of index \(\alpha\).
\end{theorem}
\begin{remark}[\textbf{Historical remarks}]\label{SmoothH}
    The smoothness of the law of \(\IPsi\) has often appeared as a by-product of the series expansions for its law; see subsection~\ref{hyperLP} below. It was informally conjectured that the law is infinitely differentiable whenever \(\overline{\Pi}(0) = \infty\) or \(\sigma^2 > 0\). Somewhat surprisingly, this turns out to be true in all cases except when \(\xi\) is a compound Poisson process with positive drift. In that situation, \(\NPs\) is finite. For instance, if \(\Psi(z) = z - q\) with \(q > 0\) (i.e., \(\xi_t = t\) killed at \(\mathbf{e}_q\)), then 
    \[
       \PIp(x) = 1 - (1 - x)^q,
    \]
    showing that smoothness can indeed fail unless \(q \in \Nb\); see \cite[Rem.~2.5]{PatieSavov2018}.
\end{remark}
\begin{remark}[\textbf{Further comments}]\label{SmoothF}
From the proof of \cite[Thm~2.4(5)]{PatieSavov2018}, which relies on a classical analytic fact from Mellin inversion, we know that
\[
\limsupi{|b|} \frac{\ln\abs{\MPsi(a+ib)}}{\ln |b|} \leq -\theta < 0
\]
implies that \(\pPs\) is analytic in the cone \(\{z \in \Cb : \abs{\arg z} < \theta\}\). From Theorem~\ref{MDecay1}, it follows, for example, that if \(\xi\) is symmetric or creeps downwards, then \(\pPs \in \Att_{\intervalOI}\). Analyticity in a cone is also guaranteed in cases where \(\phi_+\) satisfies certain conditions; since these conditions are not particularly explicit, we simply refer to \cite[Cor.~3.6]{MinSav_2023} for more information. The case of a subordinator is covered in Theorem~\ref{subAsymp} below.
\end{remark}
\begin{proof}
    The only part not explicitly established in the literature is the claim on Hölder continuity. It follows from the identity \(\Mcc_{\PIp}(z) = \MPsi(z)/z,\ z \in \Cb_{(-1,0)}\); see \cite[(7.12)]{PatieSavov2018}, which shows that \(\abs{\Mcc_{\PIp}(a+ib)}\), for \(a \in (-1,0)\), decays at rate \(1+\NPs\). Fix \(x\) and consider \(y = (1+d)x\) with \(d \in (-1/2,1/2)\). By Mellin inversion along \(a+ib,\ a \in (-\alpha,0)\), we obtain
    \begin{equation*}
        \begin{split}
            \frac{\abs{F(x)-F((1+d)x)}}{\abs{dx}^\alpha} 
            &\leq \frac{1}{2\pi} \IntII \frac{|{x^{-a-ib}-(x(1+d))^{-a-ib}}|}{\abs{dx}^\alpha}\,\abs{\Mcc_{\PIp}(a+ib)}\,\D b \\
            &= \frac{x^{-a}}{2\pi \abs{dx}^\alpha} \IntII |(1+d)^{-a-ib}-1|\,\abs{\Mcc_{\PIp}(a+ib)}\,\D b \\
            &\leq \frac{1}{\pi \abs{d}^\alpha \abs{x}^{a+\alpha}}
                 \IntII \left(|{(1+d)^{-ib}-1}| + \abs{(1+d)^a-1}\right)\abs{\Mcc_{\PIp}(a+ib)}\,\D b.
        \end{split}
    \end{equation*}
    For small \(d\), we have \(\abs{(1+d)^a-1}\leq C d\) for some constant \(C \in \Rb\), and
    \[
        \abs{(1+d)^{-ib}-1} = \abs{e^{-ib\ln(1+d)}-1} \leq 2^{1-\alpha}\min\curly{|bd|^\alpha,\,2^\alpha}.
    \]
    Substituting these estimates into the integral above and using that
    \[
        \IntII |b|^\alpha \abs{\Mcc_{\PIp}(a+ib)}\,\D b < \infty,
    \]
    since \(\alpha-(\NPs+1)<-1\), yields the claim.
\end{proof}

\subsubsection{Support}\label{SupportSec}

The results concerning the support of \(\IPsi\) are exhaustive and elementary. While partial results are scattered throughout the literature, we collect them here; see Thm.~2.4 in \cite{PatieSavov2018}.

\begin{theorem}[Thm.~2.4 in \cite{PatieSavov2018}]\label{support}
    Let \(\Psi\) be the \LLK exponent of a potentially killed \LLP \(\xi\) such that \(\phi_-(0) > 0\). Then the support of \(\IPsi\) is described as follows:
    \begin{enumerate}
        \item If \(\xi\) is a potentially killed subordinator, i.e., \(\phi_-(z) \equiv \phi_-(\infty)\), then
        \begin{equation}\label{supp1}
            \supp(\IPsi) = \lbbrbb{0,\,\frac{1}{\phi_-(\infty)\dbf_+}},
        \end{equation}
        unless \(\xi\) is pure drift, in which case \(\supp(\IPsi) = \curly{\frac{1}{\phi_-(\infty)\dbf_+}}\).
        \item If \(\xi\) is not a subordinator and \(\phi_+(z) = \dbf_+ z\), then
        \begin{equation}\label{supp2}
            \supp(\IPsi) = \lbbrb{\frac{1}{\phi_-(\infty)\dbf_+},\,\infty},
        \end{equation}
        with the convention \(1/\infty = 0\).
        \item In all other cases,
        \begin{equation}\label{supp3}
            \supp(\IPsi) = \lbbrb{0,\,\infty}.
        \end{equation}
    \end{enumerate}
\end{theorem}

\subsubsection{Factorisations}\label{FactorSec}
The form of \(\MPsi\), namely
\[
\MPsi(z) = \phi_-(0) \frac{\Gamma(z)}{\Wpp(z)} \Wpn(1-z)
\]
suggests a universal factorisation of \(\IPsi\) once one recognizes \(\Gamma(z)/\Wpp(z)\) and \(\phi_-(0)\Wpn(1-z)\)
as Mellin transforms of suitable random variables. In this section, \(\times\) stands for the product of independent random variables.

\begin{theorem}[Thm.~2.22 in \cite{PatieSavov2018}; Thm.~2.5 in \cite{ArRiv23}]\label{Factor}
    Let \(\Psi\) be a \LLK exponent of a potentially killed \LLP \(\xi\) such that \(\phi_-(0) > 0\). Then
    \begin{equation}\label{FactorEq}
        \IPsi \stackrel{d}{=} I_{\phi_+} \times X_{\phi_-} \stackrel{d}{=} I_{\phi_+} \times Y^{-1}_{\phi_-} \times e^{-\inf_{s \geq 0} \xi_s},
    \end{equation}
    where \(I_{\phi_+}\) is the exponential functional of the ascending ladder height process (a subordinator), and
    \[
        \Pb({X_{\phi_-} \in \D x}) = \phi_-(0)\, x\, \Pb({Y^{-1}_{\phi_-} \in \D x}), 
        \qquad 
        \Wpn(z+1) = \Eb[{Y_{\phi_-}^z}],\ \Re(z) > 0.
    \]
    Furthermore,
    \begin{equation}\label{FactorEq1}
        \IPsi \stackrel{d}{=} \bigotimes_{k=0}^\infty C_{\Psi}(k) \left( \Bc_k X_\Psi \times \Bc_{-k} Y_{\Psi} \right),
    \end{equation}
    where 
    \begin{equation}\label{XYPSI}
        \begin{split}
            \Pbb{X_\Psi \in \D x} &= \frac{1}{\phi_+(1)} \left( \overline{\mu}_+(-\ln x)\, \D x + \phi_+(0)\, \D x + \dbf_+\, \delta_1(\D x) \right) \quad \text{on } (0,1];\\
            \Pbb{Y_\Psi \in \D x} &= \phi_-(0) U_-(\D \ln x) \quad \text{on } (1,\infty),
        \end{split}
    \end{equation}
    and \(U_-\) is the potential measure of the descending ladder height process associated with \(\phi_-\).
    In \eqref{FactorEq1}, we also have
    \begin{equation}\label{CPSI}
        C_\Psi(0) := e^{\gamma_{\phi_+} + \gamma_{\phi_-} - \gamma}, \qquad
        C_\Psi(k) := \exp\!\left(\frac{1}{k} - \frac{\phi'_+(k)}{\phi_+(k)} - \frac{\phi'_-(k)}{\phi_-(k)}\right), \quad k \geq 1,
    \end{equation}
    where \(\gamma_{\phi_\pm}\) are the constants in the product representations \eqref{WWrep} of \(W_{\phi_\pm}\), and \(\gamma\) is the Euler–Mascheroni constant. Also, \(\Bc_x,\, x \in \Rb\), denotes the size-biased operator: if \(\Ebb{X^x} < \infty\) and \(\Pbb{X > 0} = 1\), then \(\Ebb{f(\Bc_x X)} = \Ebb{X^x f(X)}/\Ebb{X^x}\) for any bounded continuous function \(f\).
\end{theorem}

\begin{remark}[\textbf{Historical remarks}]\label{FactorH}
    The first relatively general factorisation result in this area states that, for any Bernstein function \(\phi\),
    \begin{equation}\label{FactE}
         \mathbf{e} \stackrel{d}{=} I_\phi \times Y_\phi,
    \end{equation}
    where \(\mathbf{e} \sim \mathrm{Exp}(1)\) and \(Y_\phi\) (which defines \(\Wphi\)) is for this purpose called the residual random variable. This result was first established in \cite{Bertoin-Yor-2001, BerYor02} and extended to killed subordinators in \cite[Thm.~2]{BerYor05}. It has since served as a crucial starting point in the development of the spectral theory of non-self-adjoint generalized Laguerre semigroups; see \cite{PatSav21}, and has also played a role in particular examples, such as those studied in \cite{LoePatSav19,PatSav17,PaSaZh19}.

    A factorisation of the type \eqref{FactorEq} in a more general setting was obtained in \cite{PardoPatieSavov2012, PatieSavov2012, PatieSavov2013}, where \(X_{\phi_-}\) was identified with \(I_\psi\), with \(\psi(z) = -z \phi_-(z)\) being the \LLK exponent of a spectrally negative \LL process. This requires certain assumptions on \(\xi\), for instance the \LL measure \(\Pi\) to possess a non-decreasing density on \((-\infty,0)\).
\end{remark}

The factors in \eqref{FactorEq1} can be further specified when \(\xi\) is a meromorphic \LL process; see Subsection~\ref{merLP} for the definition. In this case, we have the following result:

\begin{theorem}[Thm.~3.1 in \cite{HackKuz14}]\label{FactorMerLP}
    Let \(\xi\) be a meromorphic \LL process with \LLK exponent \(\Psi\) such that \(q = -\Psi(0) > 0\). Let \(\chi = (\chi_n(q))_{n \geq 1}\), \(\rho = (\rho_n)_{n \geq 1}\), \(\widetilde{\chi} = (1 + \hat{\chi}_n(q))_{n \geq 1}\), and \(\widetilde{\rho} = (1 + \hat{\rho}_{n-1})_{n \geq 1}\). Then
    \begin{equation}\label{FactorMerLPEq}
        \IPsi \stackrel{d}{=} C\, X(\widetilde{\rho},\widetilde{\chi}) \times \frac{1}{X(\rho, \chi)},
    \end{equation}
    where \(C = C(q)\) is defined below~\eqref{MerExp}, and for any two interlacing strictly increasing sequences \((\alpha, \beta)\), i.e., \(0 < \alpha_1 < \beta_1 < \alpha_2 < \beta_2 < \dots\), such as \((\widetilde{\rho},\widetilde{\chi})\) and \((\rho, \chi)\) in this case,
    \[
        X(\alpha, \beta) \stackrel{d}{=} \bigotimes_{n \geq 1} \frac{\beta_n}{\alpha_n}\,\mathrm{Beta}(\alpha_n, \beta_n - \alpha_n),
    \]
    where \(\mathrm{Beta}(\cdot, \cdot)\) denotes a Beta random variable with the respective parameters.
\end{theorem}

\begin{remark}[\textbf{Historical remark}]\label{FactorMerLPH}
    Special cases of Beta products of the type in \eqref{FactorMerLPEq} appear in the literature. For instance, \cite{LetSim19} considers functionals of stable processes of the form
    \[
        A = \int_0^T (\xi_s)^q \, \D s,
    \]
    where \(\xi\) is an \(\alpha\)-stable \LL process, \(T\) is its first passage time below zero, and \(q \geq 0\). It turns out that \(A\) is in fact the exponential functional of a specific meromorphic \LL process, and the moments of \(A\) are given by a particular case of \(\MPsi\), so that the factorisation \eqref{FactorMerLPEq} holds. Further examples of infinite Beta products can be found in \cite[Section~3]{Simon_23}.
\end{remark}

We conclude this part with a useful observation. Recall that a Bernstein function \(\phi\) is called a \emph{special Bernstein function} if there exists another Bernstein function \(\phi_S\) satisfying \(\phi(z)\phi_S(z) = z, \Re(z)\geq 0\); see \cite[Chapter~11]{SchVon2020} for more details. From \eqref{FactE}, it follows that 
\begin{equation}\label{FactG}
    \Gamma(z) = \Wphi(z) W_{\phi_S}(z),
\end{equation}
so that \( Y_\phi \stackrel{d}{=} I_{\phi_S} \), and equivalently \eqref{FactG} becomes
\begin{equation}\label{FactY}
    \mathbf{e} \stackrel{d}{=} I_\phi \times I_{\phi_S}.
\end{equation}
Among these formulations, the most useful is \eqref{FactG}; see Remark~\ref{SNSAsympF}. It also provides a way to study \(\Gamma/\Wphi\) through \(W_{\phi_S}\), and conversely.

\subsubsection{Some expressions}\label{SeriesSec}
In this part, we show examples where the law of \(\IPsi\) can be computed in closed form or developed as a series. We begin with the classical case:

\begin{theorem}[\cite{Duf90}; Thm.~2 in \cite{Yor92}]\label{LawBM}
    Let \(\xi = (\xi_t)_{t \geq 0} = 2 (B_t + \mu t)_{t \geq 0}\), where \(B = (B_t)_{t \geq 0}\) is an unkilled standard Brownian motion and \(\mu > 0\). Then 
    \[
        \IPsi \stackrel{d}{=} (2\Gamma_{\mu,1})^{-1},
    \]
    where \(\Gamma_{a,b}\), \(a,b>0\), denotes a \(\mathrm{Gamma}(a,b)\) random variable. 

    If \(B\) is killed at rate \(q>0\) and \(\mu\) is arbitrary, then 
    \[
        \IPsi \stackrel{d}{=}  \mathrm{Beta}(1,\alpha) \times (2\Gamma_{\beta,1})^{-1},
        \qquad \alpha := \frac{-\mu + \sqrt{2q^2 + \mu^2}}{2}, \quad \beta := \alpha + \mu.
    \]
\end{theorem}

\begin{remark}[\textbf{Further comments}]\label{LawBMF}
    For general Brownian motion, one can use \eqref{MPsiBM} to recognise
    \[
        \frac{\Gamma(z)\Gamma(1-b_-)}{\Gamma(z-b_-)} 
        \quad \text{as the Mellin transform of } \mathrm{Beta}(1,-b_-),
    \]
    and
    \[
        b_+ \frac{\sigma^{2-2z}}{2^{1-z}} \frac{\Gamma(1-z+b_+)}{\Gamma(1+b_+)} 
        = \left(\tfrac{2}{\sigma^2}\right)^{z-1} 
        \frac{\Gamma(1-z+b_+)}{\Gamma(b_+)},
    \]
    which is the Mellin transform of \(\bigl(\tfrac{2}{\sigma^2}\Gamma_{b_+,1}\bigr)^{-1}\).  
    Therefore,
    \[
        \IPsi \stackrel{d}{=} \mathrm{Beta}(1,-b_-) \times 
        \left(\tfrac{2}{\sigma^2}\Gamma_{b_+,1}\right)^{-1}.
    \]
\end{remark}

We proceed with a general result concerning three classes of \LL processes that appear in the literature, with the exception of the full generality of~\eqref{seriesOneSidedEq1}.
\begin{theorem}[Cor.~2.1 in \cite{PardoPatieSavov2012}; Cor.~1.3 in \cite{PatieSavov2012}; Thm.~2.3 in \cite{Patie2012}]\label{seriesOneSided}
    Let \(\Psi\) be the \LLK exponent of a spectrally positive \LL process, that is, \(\Pi((-\infty,0))=0\), such that \(\phi_-(0)>0\). Let \(\gamma\) denote the unique negative root of \(\Psi(z) = 0\). Then
    \begin{equation}\label{seriesOneSidedEq}
        \pPs(x) = \frac{x^{\gamma-1}}{\Gamma(-\gamma)} \IntOI e^{-{y}/{x}} y^{-\gamma} f_{\phi_+}(y)\, \D y 
        = \sum_{n=0}^\infty (-1)^n x^{\gamma-1-n} \frac{\Gamma(n-\gamma+1)}{\Gamma(-\gamma) \Wpp(n-\gamma+1)},
    \end{equation}
    where \(f_{\phi_+}\) is the density of the exponential functional of the subordinator with Laplace exponent \(\phi_+\) (that is, \(I_{\phi_+}\)), and \(x^{-\gamma-1}\pPs(x^{-1})\) is completely monotone on \(\Rp\).
    
    Next, let \(\Psi\) with \(q = -\Psi(0) > 0\) be the \LLK exponent of the negative of a killed subordinator with drift \(\dbf_- \geq 0\). Then, for any \(x > 0\), with \(X_{\phi_-}\) defined in \eqref{FactorEq},
    \begin{equation}\label{seriesOneSidedEq1}
    \begin{split}
        &\pPs(x) = \IntOI e^{-{x}/{v}} f_{X_{\phi_-}}(v)\,\frac{\D v}{v},\\
        &\pPs(x) = q + q \sum_{n=1}^\infty \frac{\prod_{j=1}^n \Psi(j)}{n!}x^n, \quad \text{on } x < 1/\dbf_-,
    \end{split}
    \end{equation}
    and \(\pPs\) is completely monotone on \(\Rp\).
    
    Finally, let \(\Psi\) with \(q = -\Psi(0) > 0\) be the \LLK exponent of a spectrally negative \LL process, that is, \(\Pi(\Rp) = 0\), and let \(\gamma_\Psi\) be the unique positive solution to \(\Psi(z) = 0\). Then
    \[
        \IPsi \stackrel{d}{=} \mathrm{Beta}^{-1}(1, \gamma_\Psi) \times I_{\psi},
    \]
    where \(\psi(z) = z\phi_-(z)\) is the \LLK exponent of a spectrally negative process that drifts to infinity. If, in addition, \(\Psi(-1)\leq 0\), then \(\pPs\) is non-increasing.
\end{theorem}

\begin{remark}[\textbf{Historical remarks}]\label{seriesOneSidedH}
    In the case when \(\xi\) is a negative subordinator, the fact that \(\pPs\) is completely monotone, as follows from \eqref{seriesOneSidedEq1}, has been reported in \cite[Cor.~2.2]{PaRiSch13}.
\end{remark}
\begin{remark}[\textbf{Further comments}]\label{seriesOneSidedF}
    Apart from \cite{Patie2012}, we are not aware of results of such generality for well-known classes of \LL processes. We also note that the complete monotonicity of \(\pPs\) implies that \(\IPsi\) is infinitely divisible; see \cite[Thm~51.6]{Sato99}.
\end{remark}
\begin{proof}
Let \(\xi\) be the negative of a killed subordinator with \( -\Psi(0) = q \). We note that \cite[Cor.~1.3(ii)]{PatieSavov2012} assumes that the \LL measure has a non-decreasing density on \((-\infty,0)\) (equivalently, in their setting where \(\IPsi = \int_0^\infty e^{\xi_s}\D s\), this corresponds to \(\Pi\) having a non-increasing density on \(\Rp\)), but this condition is in fact superfluous. Under our assumptions, \(\Psi(z) = -\phi_-(z)\), so that \(\phi_+(z) \equiv 1\). Therefore, from \eqref{MPsiAn} we obtain
\[
    \MPsi(z) = q \Gamma(z) \Wpn(1-z),
\]
and from \eqref{FactorEq} it follows that \(\IPsi \stackrel{d}{=} \mathbf{e} \times X_{\phi_-}\), where \(\mathbf{e} \sim \mathrm{Exp}(1)\).
Hence,
\begin{equation*}
    \pPs(x) = \IntOI e^{-{x}/{v}} f_{X_{\phi_-}}(v)\,\frac{\D v}{v}.
\end{equation*}
The second relation in \eqref{seriesOneSidedEq1} follows by expanding the exponential function into a Taylor series, using that 
\(\Eb[X^z_{\phi_-}] = q \Wpn(1-z)\) for \(\Re(z)<1\), and observing that 
\(\prod_{j=1}^n\Psi(j) = \bo{\dbf_-^n n!}\) if \(\dbf_- > 0\), while
\[
    \abs{\prod_{j=1}^n\Psi(j)} = \prod_{j=1}^n \phi_-(j) 
    = n! \prod_{j=1}^n \IntOI e^{-jy}\,\bar{\mu}(y)\,\D y = \so{n!},
\]
when \(\dbf_- = 0\), since 
\(\phi_-(x) = x \IntOI e^{-yx}\,\bar{\mu}(y)\,\D y\) and 
\(\limi{x}\IntOI e^{-yx}\,\bar{\mu}(y)\,\D y = 0\).
\end{proof}

We proceed with a result on exponential functionals of increasing compound Poisson processes.  
In this case, the form of \(\pPs\) is somewhat unusual.

\begin{theorem}[Prop.~5 in \cite{GuRoZw04}]\label{seriesCPP}
    Let \(\Psi\) be the \LLK exponent of a non-decreasing compound Poisson process such that 
    \(\Pi(\D x) = 0\) for \(x \leq \varepsilon\) for some \(\varepsilon > 0\). 
    Let \(S_n\) denote the sum of the first \(n\) jumps of this process, and let \(\beta \in (0,1)\) be the intensity parameter of the jump arrivals. Then
    \begin{equation}\label{CPPden}
        \pPs(x) 
        = C \sum_{n \geq 0} 
        \Ebb{\prod_{k=1}^n \frac{1}{1 - \beta^{-S_k}} \, \beta^{-S_n} \, e^{-\beta^{-S_n} x}},
        \quad 
        \text{where} \quad 
        C := \Ebb{\prod_{n=1}^{\infty} \frac{1}{1 - \beta^{S_n}}}.
    \end{equation}
\end{theorem}

We now collect results on series expansions for densities of exponential functionals. 
Our aim is to outline the main forms of the series while referring to the respective papers for details. 

We recall first the hypergeometric \LL processes; see subsection~\ref{hyperLP}, where these processes are defined via their \LLK exponent in~\eqref{hyper}. 
We also include \LLPs with jumps of rational transform, introduced at the end of subsection~\ref{merLP}.
\begin{theorem}[Thm.~4 in \cite{KuzPar13}; Prop.~3 in \cite{Kuz2012}]\label{seriesThm}
    Let \(\Psi\) be the \LLK exponent of a hypergeometric \LL process \(\xi\), see~\eqref{hyper}.
     For $\alpha>0$, set
   $
        \Psi_\alpha(z):=\Psi(\alpha z),
    $
    so that $\Psi_\alpha$ is the \LLK exponent of $(\alpha \xi_t)_{t\geq 0}$ and
    \[
         I_{\Psi_\alpha}= \int_0^\infty e^{-\alpha \xi_t}\,\D t, \qquad \alpha > 0.
    \]
    Assume that \(\alpha\) is neither rational nor a Liouville number; see \cite[Sec.~4]{KuzPar13} for details. 
    Then the density of \( I_{\Psi_\alpha}\) is given by
    \begin{equation}\label{hyperDen}
      f_{\Psi,\alpha}(x) =
      \begin{cases}
           \displaystyle \sum_{n=0}^\infty a_n x^n + \sum_{m,n=0}^\infty b_{m,n}\, x^{(m+1-\beta+\gamma)\delta+n}, & \gamma+\hat{\gamma}<1, \\[2ex]
           \displaystyle \sum_{m,n=0}^\infty c_{m,n}\, x^{-(m+\hat{\beta})\delta-n-1}, & \gamma+\hat{\gamma}>1,
      \end{cases} 
    \end{equation}
    where the coefficients \(a_n, b_{m,n}, c_{m,n}\) are given in \cite[Def.~2]{KuzPar13}. 
    
    If instead \(\Psi\) is the \LLK exponent of a \LLP with jumps of rational transform, then 
    \begin{equation}\label{ratioTrans}
        \pPs(x) = \frac{A}{\Gc(1)}\, 
        G^{K,\,\hat{M}+1}_{P+1,\,Q} \!\left(\,(A x)^{-1}\,;\,\mathbf{a},\,\mathbf{b}\right),
    \end{equation}
    where \(G^{K,\,\hat{M}+1}_{P+1,\,Q}\) denotes a Meijer's \(G\)-function, and the constants \(A\) and \(\Gc\) are defined in Subsection~\ref{merLP}. For further details, see \cite[Prop.~3]{Kuz2012}.
\end{theorem}
\begin{remark}[\textbf{Historical remarks}]\label{seriesThmH}
    If \(\xi\) is an \(\alpha\)-stable \LL process, then 
    \[
       \frac{1}{\sup_{s \leq 1} \xi_s} \stackrel{d}{=} \int_0^\infty e^{-\alpha \eta_t}\, \D t,
    \]
    where \(\eta = (\eta_t)_{t \geq 0}\) is a specific Lamperti stable \LL process (see, e.g., \cite[Prop.~4]{KuzPar13}), which in fact belongs to the class of hypergeometric \LL processes. This connection to exponential functionals makes it possible, when \(\alpha\) is neither rational nor a Liouville number, to derive a series representation for the density of \(\sup_{s \leq 1} \xi_s\); see \cite{HubKuz11} for earlier results that precede the general hypergeometric framework in \cite{KuzPar13}. In \cite[Thm.~2]{Kuznetsov13} it is shown that the range of admissible values of \(\alpha\) for which the series converges absolutely cannot be extended. These results significantly broaden earlier studies of the density of \(\sup_{s \leq 1} \xi_s\); see, e.g., \cite{BeDaPe08,DonSav10,Kuz11}, in particular regarding its small- and large-\(x\) asymptotics as described in Theorems~\ref{Cramer}, \ref{genSAsymp}, and~\ref{genSAsymp1}.
 Moreover, the asymptotic implications are further significantly improved by the analysis of the asymptotic behaviour of \(\pPs\) for general exponential functionals; see subsection~\ref{Lasymp}.
\end{remark}
  \begin{remark}[\textbf{Further comments}]\label{seriesThmF}
    Although expressing the density as a series is appealing, the restriction that \(\alpha\) cannot vary freely makes it more practical to express it in terms of special functions. In this direction, the recent work \cite[Sec.~6]{KarKuz24} shows that \(f_{\Psi,\alpha}\) is in fact a special case of the \(K\)-function, a natural extension of Meijer's \(G\)-function.
\end{remark}

\subsubsection{Large asymptotics of the density and tail}\label{Lasymp}
The behaviour of the tail probability \(\PIPs{x} := \Pbb{\IPsi > x}\) and the density \(\pPs(x)\) as \(x \to \infty\) is one of the most extensively studied aspects of exponential functionals of \LL processes. In what follows, we first present results in a general framework. Then, instead of offering only historical remarks, we formulate several partial results as separate statements and provide additional comments on the main methodologies used in this area.

Throughout this section we use the standard notation \(f \simi g\) to mean that \(\limi{x} f(x)/g(x) = 1\). The symbols \(\so{\cdot}\) and \(\bo{\cdot}\) denote the classical asymptotic relations, for instance \(f(x) = \so{x}\) means \(\limi{x} f(x)/x = 0\).

The first result applies to all exponential functionals. Recall that \(\bar{\ak}_-^\Psi= \bar{\ak}_- \leq 0\), see~\eqref{aPsi}.

\begin{theorem}[Thm.~2.11 in \cite{PatieSavov2018}]\label{genLAsymp}
    Let \(\Psi\) be the \LLK exponent of a potentially killed \LLP \(\xi\) with \(\phi_-(0) > 0\). Then
    \begin{equation}\label{genLAsympeq}
        \lim_{x \to \infty} \frac{\ln \PIPs{x}}{\ln x} = \bar{\ak}_-^\Psi,
    \end{equation}
    where \(\bar{\ak}_-^\Psi = \ak_-^\Psi = -\infty\) if and only if \(\Psi(z) = -\phi_-(z)\), i.e., when the associated \LLP \(\xi\) is a potentially killed subordinator.
\end{theorem}
\begin{remark}[\textbf{Historical remarks}]\label{genLAsympH}
    As noted in \cite{ArRiv23}, it follows from \cite[Thm.~4.1]{GolGru96} that if \(\xi\) is not a subordinator, then 
    \(\lim_{x \to \infty} \ln \PIPs{x} / \ln x > -\infty\). 
    This observation arises from the fact that \(\IPsi\) is a perpetuity, so general results for affine equations apply. 
    However, the specific structure of exponential functionals of \LL processes enables sharper results, not only providing precise logarithmic asymptotics, but also describing the asymptotic behaviour on the real scale.
\end{remark}

\begin{remark}[\textbf{Further comments}]\label{genLAsympF}
    According to \eqref{aPsiSpecial}, \(\bar{\ak}_-^\Psi = 0\) if and only if \(\ak^\Psi_- = 0\), which is equivalent to \(\xi_1\) not admitting any negative exponential moments. 
    Moreover, since from \eqref{BG} 
    \[
        \phi_-(z) = \phi_-(0) + \dbf_- z + \IntOI \lbrb{1 - e^{-z y}} \mu_-(\D y),
    \]
    it follows that if \(\ak^\Psi_- = \ak_- = -\infty\), then \(\lim_{x \to -\infty} \phi_-(x) = -\infty\). Hence, \(\bar{\ak}^\Psi_- = \uk^\Psi_- = \uk_- > -\infty\), unless \(\phi_-(z) \equiv \phi_-(0)\). Thus,
    \[
        \bar{\ak}_-^\Psi = -\infty \iff \Psi(z) = -\phi_+(-z)\phi_-(0) \iff \xi \text{ is a potentially killed subordinator}.
    \]
    In the case \(\bar{\ak}_-^\Psi \in \lbrb{-\infty, 0}\), we obtain a quantifiable rate of decay at the logarithmic scale.
\end{remark}
The case \(\bar{\ak}_-^\Psi = -\infty\) in Theorem~\ref{genLAsymp}, corresponding to \(\xi\) being a subordinator, requires separate treatment. We present below the most general result available to date.

\begin{theorem}[Thm.~3.1 in \cite{MinSav_2023}; Thm.~10.1 in \cite{Haas21}]\label{subAsymp}
    Let \(\phi\) be a Bernstein function with \(\dbf = 0\), and let \(\Psi(z) = -\phi(-z)\) be the \LLK exponent of a subordinator. Suppose that
    \begin{equation}\label{condH}
        \limsupi{x} \frac{x \phi'(x)}{\phi(x)} < 1,
        \quad \text{or equivalently} \quad 
        \liminf_{x \to 0} \frac{\int_0^{2x} \bar{\mu}(y)\,\D y}{\int_0^x \bar{\mu}(y)\,\D y} > 1.
    \end{equation}
    Then \(\phi_*(x) := x / \phi(x)\) is increasing, with inverse
    \[
        \varphi_* := \phi_*^{-1} : \lbrb{(\phi'(0^+))^{-1}\ind{\phi(0)=0}, \infty} \to \intervalOI.
    \]
    For every \(n \geq 1\), the following asymptotics hold:
    \begin{equation}\label{subAsympEq}
        \pPs^{(n)}(x) \simi 
        \frac{C_\phi \,\varphi_*^n(x) \sqrt{\varphi_*'(x)}}{x^n} 
        \exp\left\{ - \int_{\phi_*(1)}^x \frac{\varphi_*(y)}{y}\,\D y \right\},
    \end{equation}
    where
    \[
        C_\phi := \frac{(-1)^n e^{-T_{\phi_*}}}{\sqrt{2\pi\,\phi_*(1)}},
        \qquad
        T_{\phi_*} := \IntOI (u - \lceil u \rceil)(1 - u + \lceil u \rceil) 
        \left[ \frac{1}{u^2} - \left( \frac{\phi'(u)}{\phi(u)} \right)^2 + \frac{\phi''(u)}{\phi(u)} \right] \D u \in \Rb,
    \]
    with \(\lceil \cdot \rceil\) denoting the ceiling function. Moreover,
    \begin{equation}\label{subAsympEqTail}
        \PIPs{x} \simi 
        \frac{e^{-T_{\phi_*}}}{\sqrt{2\pi\,\phi_*(1)}}
        \frac{x\,\sqrt{\varphi_*'(x)}}{\varphi_*(x)}
        \exp\left\{ - \int_{\phi_*(1)}^x \frac{\varphi_*(y)}{y}\,\D y \right\}.
    \end{equation}
    As a consequence, the even derivatives of \(\pPs\) are both ultimately monotone decreasing and ultimately convex, while the odd derivatives are both ultimately monotone increasing and ultimately concave. If in addition 
    \(\limsupo{x} \bar{\mu}(2x)/\bar{\mu}(x) < 1\), then there exists \(\varepsilon > 0\) such that \(\pPs\) is analytic in the cone 
    \(\{ z \in \Cb : \Re(z) > 0,\, |\arg z| < \varepsilon \}\).
\end{theorem}
\begin{remark}[\textbf{Historical remarks}]\label{subAsympH}
   At the logarithmic scale and under the same conditions, the asymptotic relation \eqref{subAsympEqTail} has already appeared in \cite[Prop.~3.1]{HaRiv12}. We also note that \cite[Thm.~10.1]{Haas21}, which covers the case \(n=0\), does not provide an explicit expression for \(C_\phi\), but does establish a convergence rate of order \(\bo{1/\varphi_*(x)}\). The method in \cite{Haas21} differs from that in \cite{MinSav_2023}, yet both rely on the same condition \eqref{condH}, indicating that this assumption is likely necessary for the validity of \eqref{subAsympEq}.
\end{remark}
\begin{remark}[\textbf{Further comments}]\label{subAsympF} 
If \(\dbf > 0\), then from \eqref{BG} and \eqref{B-CM} one has the limit \(\lim_{x \to \infty} (x\phi'(x))/\phi(x) = 1\), so condition~\eqref{condH} does not hold. In this case, since \(\xi_t \geq \dbf t\), we have \(\IPsi \leq \dbf^{-1}\), meaning that \(\IPsi\) has finite support. Thus, a transition in the asymptotic behaviour occurs when \(\limsupi{x} x\phi'(x)/\phi(x) = 1\). If \(\varphi_*(x)\) denotes the inverse of \(x/(\phi(x) - q - \dbf x)\), then \cite[Prop.~3.12]{HaRiv12} shows that
\[
    \ln \PIPs{x} \sim \dbf \int_a^x \varphi_*\!\left( \tfrac{u}{1 - \dbf u} \right) \D u,
    \qquad
    \text{as \(x \uparrow \dbf^{-1}\), with \(a < \dbf^{-1}\) fixed.}
\]
\end{remark}

It is useful to state separately the case of a non-decreasing compound Poisson process, since the result shows that the asymptotics coincide with that of the tail of an exponential distribution.

\begin{corollary}[Cor.~3.4 in \cite{MinSav_2023}]\label{CPPAsymp}
    Let \(\phi\) be a Bernstein function with \(\dbf = 0\), and let \(\Psi(z) = -\phi(-z)\) be the \LLK exponent of a non-decreasing compound Poisson process. Let \(\mu\) denote the Lévy measure of \(\phi\). If \(\int_0^1 \mu(\D v)/v = \infty\), then
     \begin{equation}\label{CPPAsympEq}
        \pPs^{(n)}(x) \simi C e^{-\phi(\infty)x}.
    \end{equation}
    If instead \(\int_0^1 \mu(\D v)/v < \infty\), then
    \begin{equation}\label{CPPAsympEq2}
        \pPs^{(n)}(x) \simi C e^{-\phi(\infty)(x + \so{x})}.
    \end{equation}
    In both cases, \(C\) is an explicit constant depending on \(\phi(\infty)\), \(\phi(1)\), \(\varphi_*(1)\), and \(T_{\phi_*}\); see \cite[Cor.~3.4]{MinSav_2023} for details.
\end{corollary}

The next set of relatively general results stems from the three-term factorisation in \eqref{FactorEq},
\[
    \IPsi \stackrel{d}{=} I_{\phi_+} \times Y^{-1}_{\phi_-} \times e^{-\inf_{s \geq 0} \xi_s},
\]
which supports the conjecture, first proposed in \cite{Rivero12}, that
\begin{equation}\label{tailConj} 
    \PIPs{x} \simi c\,\Pbb{e^{-\inf_{s \geq 0} \xi_s} > x},
\end{equation}
for some constant \(c>0\). Note that \cite{ArRiv23, Rivero12} consider \(\IntOI e^{\xi_s}\D s\), so the sign must be changed when referring to these works in our manuscript. The conjecture is further reinforced by the fact that \(I_{\phi_+}\) has finite moments of all orders, and the moments of \(Y^{-1}_{\phi_-}\) are not of smaller order than those of \(e^{-\inf_{s \geq 0} \xi_s}\). 

To state these results, we briefly recall the class of convolution equivalent Lévy processes for which \eqref{tailConj} is known to hold. We say that \(-\xi \in \mathcal{S}_\gamma\) for some \(\gamma \geq 0\), or equivalently that \(-\xi\) is convolution equivalent with index \(\gamma\), if for all \(y \in \Rb\),
\[
    \lim_{x \to \infty} \frac{\Pbb{-\xi_1 > x - y}}{\Pbb{-\xi_1 > x}} = e^{\gamma y},
    \quad \text{and} \quad
    \lim_{x \to \infty}\frac{\int_\Rb \Pbb{-\xi_1 > x - y}\Pbb{-\xi_1 \in \D y}}{\Pbb{-\xi_1 > x}}
    = 2\int_\Rb e^{\gamma y}\Pbb{-\xi_1 \in \D y}.
\]
A result of \cite{Pakes04} shows that these relations are equivalent to the same statements with \(\Pbb{-\xi_1 \in \D x}\) replaced by the truncated Lévy measure \(\Pi_-(\D x)\ind{x > 1} := \Pi(-\D x)\ind{x > 1}\).

\begin{theorem}[Thm.~2.9 in \cite{ArRiv23}]\label{RivConvo}
    Let \(\Psi\) be the \LLK exponent of a potentially killed \LLP \(\xi\) such that \(\phi_-(0) > 0\). For \(x > 0\), set \(\overline{\Pi}_-(x) := \Pi((-\infty, -x))\). Then the following statements hold:
    \begin{enumerate}[(i)]
        \item If the law of \(\xi\) is not supported on a lattice and there exists \(\theta < 0\) such that \(\Psi(\theta) = 0\) and \(|\Psi'(\theta)| < \infty\), then
        \begin{equation}\label{CramerEq}
            \lim_{x \to \infty} x^{-\theta} \PIPs{x} = C_\theta \in (0, \infty).
        \end{equation}
        \item Assume \(\Pi((-\infty, 0)) > 0\) (i.e., \(\xi\) has negative jumps) and \(-\xi_1 \in \mathcal{S}_0\). 
        \begin{enumerate}[(a)]
            \item If \(\Psi(0) < 0\), then
            \begin{equation}\label{RivConvoEq1}
                \PIPs{x} \simi \overline{\Pi}_-(\ln x).
            \end{equation}
            \item If \(\Psi(0) = 0\) and \(\Ebb{\xi_1} \in (0, \infty)\), then
            \begin{equation}\label{RivConvoEq2}
                \PIPs{x} \simi \frac{1}{\Ebb{\xi_1}} \int_{\ln x}^\infty \overline{\Pi}_-(s)\, \D s.
            \end{equation}
        \end{enumerate}
        \item If \(-\xi \in \mathcal{S}_{-\alpha}\) for some \(\alpha < 0\), and \(\Psi(\alpha) < 0\), then
        \begin{equation}\label{RivConvoEq3}
           \PIPs{x} \simi C_\alpha \int_{\ln x}^\infty \overline{\Pi}_-(s)\, \D s.
        \end{equation}
    \end{enumerate}
\end{theorem}
\begin{remark}[\textbf{Historical remarks}]\label{RivConvoH}
    The relations \eqref{CramerEq}, \eqref{RivConvoEq1}, \eqref{RivConvoEq2}, and \eqref{RivConvoEq3} confirm the conjecture of Rivero stated in \eqref{tailConj}. The constants in \eqref{RivConvoEq1} and \eqref{RivConvoEq3} are specified in \cite[Thm.~2.9]{ArRiv23}. This theorem is the culmination of a sequence of works: see \cite{Riv05} and \cite[Thm.~2.11]{PatieSavov2018} for part~(i); \cite[Thm.~4.1]{Maulik-Zwart-06} for part~(ii) with \(\Psi(0) = 0\); and \cite[Thm.~1]{Rivero12} for part~(iii) with \(\Psi(0) = 0\) under the additional assumptions \(\alpha \in (0,1)\) and \(\Ebb{|\xi_1|} < \infty\).
\end{remark}

\begin{remark}[\textbf{Further comments}]\label{RivConvoF}
    Theorem~2.7 in \cite{ArRiv23} establishes the validity of \eqref{tailConj} when \(\PIPs{x}\) is regularly varying with index \(\alpha \geq 0\), which already covers a broad range of cases. However, beyond the class of convolution equivalent processes, no more general results appear to be available in the literature. We also note that case~(i) requires the classical Cramér condition.
\end{remark}
The results of Theorem~\ref{RivConvo}(i) can be further strengthened under mild additional assumptions. To this end, we introduce the notion of a \textit{weak non-lattice} \LL process. We say that a \LLP is \textit{weak non-lattice} if and only if \(\uk^\Psi_- \in \lbrb{-\infty,0}\) and there exists \(k \geq 0\) such that
\[
    \liminf_{|b| \to \infty} |b|^k \Psi\lbrb{\uk^\Psi_- + ib} > 0.
\]
As pointed out in \cite{DSTW24}, there exist examples (see \cite[Ex.~41.3]{Sato99}) where this limit fails for \(k = 0\), and the construction therein suggests that it may in fact fail for all \(k \geq 0\).
\begin{theorem}[Thm.~2.11(2) in \cite{PatieSavov2018}]\label{Cramer}
   Let \(\Psi\) be the \LLK exponent of a potentially killed \LLP \(\xi\) such that \(\phi_-(0) > 0\). Assume further that \(\uk^\Psi_- \in \lbrb{-\infty,0}\), that the right derivative at \(\uk^\Psi_-\) is finite, i.e. \(|\Psi'(\uk^\Psi_-+)| < \infty\), and that \(\xi\) is \textit{weak non-lattice}. Then
   \begin{equation}\label{CramerEq1}
       \PIPs{x} \simi 
       \frac{\phi_-(0)\,\Gamma(-\uk^\Psi_-)\,\Wpn(1+\uk^\Psi_-)}{\phi'_-(\uk^\Psi_-+)\,\Wpp(1-\uk^\Psi_-)}
       x^{\uk^\Psi_-}.
   \end{equation}
   If, in addition, \(|\Psi''(\uk^\Psi_-+)| < \infty\) and \(\NPs > 1\), then for all \(n \leq \lceil \NPs \rceil - 2\),
   \begin{equation}\label{CramerEq2}
       \pPs^{(n)}(x) \simi 
       (-1)^n \frac{\phi_-(0)\,\Gamma(n+1-\uk^\Psi_-)\,\Wpn(1+\uk^\Psi_-)}{\phi'_-(\uk^\Psi_-+)\,\Wpp(1-\uk^\Psi_-)}
       x^{\uk^\Psi_- - n - 1}.
   \end{equation}
   Moreover, if \(\NPs = \infty\), then all even derivatives are ultimately positive, decreasing and convex, while all odd derivatives are ultimately negative, increasing and concave.
\end{theorem}
\begin{remark}[\textbf{Historical remarks}]\label{CramerH}
    Relation~\eqref{CramerEq1} was first established in \cite[Lem.~4]{Riv05}, but only up to an unspecified constant. The exact value of this constant was later determined in \cite[Thm.~2.11(2)]{PatieSavov2018}. Apart from cases where \(\pPs\) admits an expansion in classical or asymptotic series (see subsection~\ref{SeriesSec}), there are relatively few results of the type~\eqref{CramerEq2}.
\end{remark}
\begin{remark}[\textbf{Further comments}]\label{CramerF}
    Unlike \cite{Riv05}, the approach in \cite{PatieSavov2018} relies on Mellin inversion to obtain refined results for \(\pPs^{(n)}\) as in~\eqref{CramerEq2}. Although Mellin inversion is a powerful method, the requirement that \(\xi\) be \emph{weak non-lattice}, while very mild, is somewhat inconvenient. Interestingly, the possibility that \(\xi\) is only \emph{weak non-lattice} creates a major obstacle to proving the uniqueness of the Wiener--Hopf factorisation \eqref{WHf} via the classical application of Liouville's theorem; see \cite{DSTW24}, where an appeal to Schwartz theory is instead required. Another minor restriction, stemming from the method, is the condition \(|\Psi''(\uk^\Psi_-+)| < \infty\), which is slightly stronger than Cramér's condition. This last restriction, however, disappears whenever \(\ak^\Psi_- < \uk^\Psi_-\).
\end{remark}

We conclude this part with results on the rate of convergence in some of the presented results. This topic has received limited attention in the literature, being explicitly addressed in \cite{ArRiv23}, and in some cases it can also be derived from the complex-analytic approach of \cite{PatieSavov2018}.
\begin{theorem}[Thm.~2.14 in \cite{ArRiv23}]\label{rateThm}
    Let \(\Psi\) be the \LLK exponent of a potentially killed \LLP \(\xi\) such that \(\phi_-(0) > 0\). 
    Assume that the measure 
    \(\IntOI e^{-t}\Pbb{\xi_t \in \D y}\D t\) is absolutely continuous on \(\Rb\), and that 
    \(\Pb(T_{(-\infty,0)} > 0) = 0\), where \(T_{(-\infty,0)} := \inf\{t>0:\xi_t<0\}\).
    Suppose further that \(\uk^\Psi_- \in (-\infty,0)\) and 
    \(\Eb[|\xi_1|^{m+1} e^{\uk^\Psi_- \xi_1}] < \infty\) for some integer \(m \geq 2\). Then
    \begin{equation}\label{rateThm_eq}
        \sup_{a \leq \delta < \lambda \leq \infty}
        \abs{ C x^{-\uk^\Psi_-} \Pbb{\IPsi \in \lbrb{x\delta, x\lambda}}
        - \int_\delta^\lambda \frac{\D y}{(-\uk^\Psi_-)y^{1-\uk^\Psi_-}} }
        = \so{(\log x)^{-m+1}},
    \end{equation}
    for some \(a > 0\), with a semi-explicit constant \(C\); see \cite[Thm.~2.4]{ArRiv23}.

    Moreover, if there exists \(\rho > 0\) such that 
    \(\Eb[e^{(\uk^\Psi_- - \rho)\xi_1}] < \infty\), then there exists 
    \(0 < \gamma < \min\{1,\rho\}\) such that
    \begin{equation}\label{rateThm_eq1}
        \sup_{a \leq \delta < \lambda \leq \infty}
        \abs{ C x^{-\uk^\Psi_-} \Pbb{\IPsi \in \lbrb{x\delta, x\lambda}}
        - \int_\delta^\lambda \frac{\D y}{(-\uk^\Psi_-)y^{1-\uk^\Psi_-}} }
        = \so{x^{-\gamma}}.
    \end{equation}
\end{theorem}
The results above can be complemented by the following new result.
\begin{theorem}\label{rateThm1}
    Let \(\Psi\) be the \LLK exponent of a potentially killed \LLP \(\xi\) such that \(\phi_-(0) > 0\), \(\uk^\Psi_- > \ak^\Psi_-\), and \(\NPs > 1\).
    Set
    \[
        \uk^\Psi_-(2):=\sup\curly{\Re z:\; z\in \Cb_{(\ak^\Psi_-,\,\uk^\Psi_-)} \ \text{and}\ \Psi(z)=0},
    \]
    and assume \(\uk^\Psi_-(2)<\uk^\Psi_-\).
    Finally, if \(\NPs=\infty\), assume that for any \(y\in(\uk^\Psi_-(2),\,\uk^\Psi_-)\) there exists \(k=k(y)\in\Nb\) such that
    \[
        \liminf_{|b|\to\infty} |b|^{k}\abs{\Psi(y+ib)}>0.
    \]
    Then, for any \(y\in(\uk^\Psi_-(2),\,\uk^\Psi_-)\), as \(x\to\infty\),
    \[
        \abs{\pPs(x)-C\,x^{\uk^\Psi_- -1}}=\so{x^{\,y-\uk^\Psi_-}},
        \qquad
        \text{where}\quad
        C:=\frac{\phi_-(0)\,\Gamma(1-\uk^\Psi_-)\,\Wpn(1+\uk^\Psi_-)}{\phi'_-(\uk^\Psi_-+)\,\Wpp(1-\uk^\Psi_-)}.
    \]
\end{theorem}
\begin{proof}
    The assumptions of this theorem allow the decay \(\NPs\) of \(\MPsi\) to be transferred to the strip \(\Cb_{\lbrb{\uk^\Psi_-(2),\,\uk^\Psi_-}}\); see \cite[Prop.~7.2]{PatieSavov2018}. 
    Then, the Mellin inversion argument in the proof of \cite[Thm.~2.11]{PatieSavov2018} applies: shifting the line of Mellin inversion picks up the pole at \(1-\uk^\Psi_-\), which yields the main asymptotic term, while the shifted line of inversion gives the rate of convergence.
\end{proof}

\subsubsection{Small asymptotics}\label{Sasymp}
In this part, we present results on the behaviour of \(\PIp(x)\) and its derivatives as \(x \to 0\). We use the notation \(f \simo g\) to mean that \(\lim_{x \to 0} f(x)/g(x) = 1\).

We start with a general result.
\begin{theorem}[Thm.~2.15 in \cite{PatieSavov2018}; Thm.~2.12(i) in \cite{ArRiv23}]\label{genSAsymp}
    Let \(\Psi\) be the \LLK exponent of a potentially killed \LLP \(\xi\) such that \(\phi_-(0) > 0\). Then
    \begin{equation}\label{genSAsympEq1}
        \limo{x}\frac{\PIp(x)}{x} = -\Psi(0) = \pPs(0^+).
    \end{equation}
    If \(\pPs\) is continuous at zero, which holds whenever \(\NPs > 1\), see~\eqref{Npsi}, then \(\limo{x}\pPs(x) = -\Psi(0)\).
\end{theorem}

\begin{remark}[\textbf{Historical remarks}]\label{genSAsympH}
    For the case \(\Psi(0)<0\), and under minor additional assumptions, \eqref{genSAsympEq1} was proved in \cite[Thm.~2.12(i)]{ArRiv23}.  
    When \(\xi\) is a subordinator with \(\Psi(0)<0\), \cite[Thm.~2.5]{PaRiSch13} establishes that \(\pPs\) is continuous at zero.
\end{remark}
The next result is fairly general, refining Theorem~\ref{genSAsymp} by providing higher-order asymptotics in the case $-\Psi(0)>0$, that is, when $\xi$ is killed.

\begin{theorem}[Thm.~2.4(4) and Cor.~2.9 in \cite{PatieSavov2018}]\label{genSAsymp1}
    Let $\Psi$ be the \LLK exponent of a potentially killed \LLP $\xi$ such that $\phi_-(0) > 0$ and $-\Psi(0) > 0$. Set
    \[
        \mathtt{N}^+ := \uk^\Psi_+\ind{\uk^\Psi_+ \in \Nb} + \lceil \ak^\Psi_+ + 1 \rceil \ind{\uk^\Psi_+ \notin \Nb}.
    \]
    Then, for any $0 \leq n < \mathtt{N}^+$, any $M \in \Nb$ with $M < \mathtt{N}^+$, and any real number $a \in ( \max \{-M-1, -\ak^\Psi_+ - 1\}, -M)$ and $x> 0$, we have
    \begin{equation}\label{genSAsymp1_1}
    \begin{split}
        \PIp^{(n)}(x) &= -\Psi(0)\sum_{k=\max\{1, n\}}^{M} \frac{ \prod_{j=1}^{k-1} \Psi(j) }{ (k-n)! } x^{k-n} \\
        &\hspace{2cm}+(-1)^{n+1} \frac{ \phi_-(0) }{ 2\pi i } \int_{a-i\infty}^{a+i\infty} x^{-z-n} \Gamma(z+n) \frac{ \Wpn(-z) }{ \Wpp(1+z) } \D z.
    \end{split}
    \end{equation}
    Therefore, if $\uk^\Psi_+ \notin \Nb$ and $\ak^\Psi_+ = \infty$, then, as an asymptotic expansion,
    \begin{equation}\label{genSAsymp1_2}
        \PIp(x) \stackrel{0}{\sim} -\Psi(0) \sum_{k=1}^{\infty} \frac{ \prod_{j=1}^{k-1} \Psi(j) }{ k! } x^k.
    \end{equation}
    This expansion cannot be a convergent series unless \(\phi_+(z)\equiv \phi_+(\infty)\) or \(\phi_+(z)=\phi_+(0)+\dbf_+z\). In these cases,
    \[    
        \PIp(x) = -\Psi(0) \sum_{k=1}^{\infty} \frac{ \prod_{j=1}^{k-1} \Psi(j) }{ k! } x^k,
    \]
    valid in the following regimes:
    \begin{enumerate}
        \item for all \(x\in \intervalOOI\), provided \(\phi_+(z)=\phi_+(\infty)\) and \(\dbf_-=0\);
        \item for all \(x<(\phi_+(\infty)\dbf_-)^{-1}\), if \(\phi_+(z)=\phi_+(\infty)\) and \(\dbf_->0\); the series is divergent for \(x>(\phi_+(\infty)\dbf_-)^{-1}\);
        \item for all \(x<(\phi_-(\infty)\dbf_+)^{-1}\), provided \(\phi_+(z)=\phi_+(0)+\dbf_+ z,\ \dbf_+>0\), and \(\phi_-(\infty)<\infty\); the series is divergent for \(x>(\phi_-(\infty)\dbf_+)^{-1}\).
    \end{enumerate}
\end{theorem}
\begin{remark}[\textbf{Historical remarks}]\label{genSAsymp1H}
    The case $\phi_+(z)\equiv \phi_+(\infty)$ is stated explicitly in \cite[Cor.~1.3]{PatieSavov2012} and appears implicitly in \cite{PaRiSch13}.
\end{remark}

\begin{remark}[\textbf{Bertoin--Yor conjecture}]\label{BYConj1}
    If $\Psi(0)=0$ and $-\ak^\Psi_+=-\infty$, then from \eqref{genSAsymp1_1} it follows that $\pPs(x)\stackrel{0}{=}\so{x^n}$ for every $n\geq 0$. To prove that $\IPsi^{-1}$ is moment-indeterminate by applying the Krein or Lin condition, one would need information on the decay of $\ln(\pPs(x^2))$ as $x\to 0$. A reformulation of the Krein condition, see \cite{Lin17}, shows that it suffices to have 
    \[
        \pPs(x)\stackrel{0}{=}\so{\exp\lbrb{-x^{-\frac14+\epsilon}}}, \qquad \epsilon>0.
    \]
    Establishing such a bound may require a saddle point method, possibly more intricate than the one developed in \cite{MinSav_2023}.
\end{remark}

The following result, valid in the case $\Psi(0)=0$, does not appear explicitly in the literature, but can be deduced from \cite[Thm.~2.4(3)]{PatieSavov2018}.
\begin{theorem}\label{genSAsymp10}
   Let $\Psi$ be the \LLK exponent of a potentially killed \LLP $\xi$ such that $\phi_-(0) > 0$ and $\Psi(0) = 0$. Then the following hold:
   \begin{enumerate}
       \item If $\ak^\Psi_+ = 0$, $\phi'_+(0+)<\infty$, and $\NPs>1$, then for $n\leq \lceil \NPs \rceil - 2$,
       \begin{equation}\label{genSAsymp10_Eq}
           \limo{x}\pPs^{(n)}(x)=0.
       \end{equation}
       \item If $-\ak^\Psi_+ < 0$, then for all $\epsilon>0$ and $n\leq \lceil \NPs \rceil - 1$,
       \begin{equation}\label{genSAsymp10_Eq1}
           \limo{x} x^{-\ak^\Psi_+ + \epsilon}\PIp^{(n)}(x)=0.
       \end{equation}
       \item If $-\ak^\Psi_+ < 0$ and $\abs{\phi_+(-\ak^\Psi_+)}<\infty$, then for all $n\leq \lceil \NPs \rceil - 1$,
       \begin{equation}\label{genSAsymp10_Eq2}
           \limo{x} x^{-\ak^\Psi_+}\PIp^{(n)}(x)=0.
       \end{equation}
   \end{enumerate}
\end{theorem}
\begin{proof}
    All items follow from Mellin inversion and the Riemann--Lebesgue theorem. The only point to note is that, in the third item, the Mellin inversion contour can be pushed to the line~\(\Cb_{-\ak^\Psi_+}\) thanks to Corollary~\ref{corMDecay} and Proposition~\ref{uniMpsi}, the latter ensuring uniform decay along strips of finite length and thereby allowing the shift of the contour to~\(\Cb_{-\ak^\Psi_+}\); see also Remark~\ref{uniMpsiF}. We briefly sketch one case, as the others follow with similar arguments \emph{mutatis mutandis}. From~\eqref{Smoothf}, with \(z=\ab\), we obtain
    \begin{equation*}
        \begin{split}
            x^a \pPs^{(n)}(x) &= \frac{(-1)^n \phi_-(0)}{2\pi i} \int_{a-i\infty}^{a+i\infty} x^{-ib} \Gamma(z+n) \frac{\Wpn(1-z)}{\Wpp(z)} \D z \\
            &= \frac{(-1)^n \phi_-(0)}{2\pi i} \int_{a-i\infty}^{a+i\infty} e^{-ib \ln x} \Gamma(z+n) \frac{\Wpn(1-z)}{\Wpp(z)} \D z,
        \end{split}
    \end{equation*}
    and since \(\Gamma(z+n)\Wpn(1-z)/\Wpp(z)\) is absolutely integrable along \(\Cb_a\), the Riemann--Lebesgue theorem applies as \(-\ln x\to\infty\).
\end{proof}

\begin{theorem}[Rem.~2.17 in \cite{PatieSavov2018}; Thm.~5.5 in \cite{PatSav21}]\label{SNSAsymp}
    Let \(\Psi\) be the \LLK exponent of a potentially killed \LLP \(\xi\) such that \(\phi_-(0) > 0\), \(\Pi\lbrb{(0,\infty)} = 0\), and either \(\sigma^2 > 0\) or \(\int_{0}^\infty \Pi((x,\infty)) \D x = \infty\); 
    that is, the process has unbounded variation and thus \(\phi_-(\infty) = \infty\), where \(\phi_-\) is the Wiener--Hopf factor as defined in~\eqref{WHf}. Assume in addition that \(\Psi(0) = 0\). Let \(\varphi_- := \phi_-^{-1} : \lbrb{\phi_-(0),\infty} \to (0,\infty)\). Then, for any \(n \geq 0\),
    \begin{equation}\label{SNSAsympEq}
        \pPs^{(n)}(x) \simo \frac{C_{\phi_-}\,\phi_-(0)}{\sqrt{2\pi}} \frac{\varphi_-^n\lbrb{1/x}}{x^{n+1}} \sqrt{\varphi'_-\lbrb{1/x}} 
        \exp\!\left(-\int_{\phi_-(0)}^{1/x} \frac{\varphi_-(v)}{v} \D v \right),
    \end{equation}
    where \(C_{\phi_-}\) is an implicit positive finite constant.
\end{theorem}
\begin{remark}[\textbf{Historical remarks}]\label{SNSAsympH}
    The result in \cite[Rem.~2.17]{PatieSavov2018} is a restatement of \cite[Thm.~5.5]{PatSav21}. The proof in \cite{PatSav21} relies on a non-classical Tauberian theorem, which accounts for the implicit nature of the constant \(C_{\phi_-}\). In fact, it concerns a density obtained as a transformation of \(\pPs\).
\end{remark}

\begin{remark}[\textbf{Further comments}]\label{SNSAsympF}
    Note that when \(\xi\) is a spectrally negative \LL process, i.e. \(\Pi\lbrb{(0,\infty)} = 0\), then \(\Psi(z)=z\phi_-(z)\) and from \eqref{MPsiAn}
    \[\MPsi(z)=\phi_-(0)\Wpn(1-z).\]
    Assume for a moment that \(\phi_-\) pertains to a special subordinator, that is there exists a Bernstein function \(\phi_S\) such that \(\phi_-(z)\phi_S(z)=z,\) see \cite{SchVon2020} for more information on special subordinators. Then clearly \(\Wpn(z)W_{\phi_S}(z)=\Gamma(z)\) and therefore
    \[\MPsi(z)=\phi_-(0)\frac{\Gamma(1-z)}{W_{\phi_S}(1-z)},\]
    where the latter is the Mellin transform of the exponential functional of the subordinator pertaining to \(\phi_S\). Given the asymptotics at infinity of the density of the latter, see \eqref{subAsympEq}, we see expected similarity to \eqref{SNSAsympEq}. However, the approach in \cite{PatieSavov2018} for obtaining \eqref{subAsympEq} requires non-classical Tauberian theorem and link to self-decomposability.
\end{remark}

The next result refines \eqref{genSAsympEq1} in some special cases and complements Theorem \ref{SNSAsymp} in the case where \(\Pi\lbrb{(0,\infty)} > 0\).
\begin{theorem}[Thm.~2.12 in \cite{ArRiv23}]\label{convoSAsymp}
     Let \(\Psi\) be the \LLK exponent of a potentially killed \LLP \(\xi\) such that \(\phi_-(0) > 0, \Psi(0)=0\), and \(\Pi\lbrb{(0,\infty)} > 0\). Let for some \(\alpha\geq 0\)
     \[\limi{x}\frac{\Pi\lbrb{(x+y,\infty)}}{\Pi\lbrb{(x,\infty)}}=e^{-\alpha y}\quad
     \text{for any }y\in\Rb.\]
     If \(\alpha>0\), assume in addition that \(\Psi(\alpha)<0\), or equivalently \(\bar{\ak}^{\Psi}_+=\ak^\Psi_+\) and \(\uk^\Psi_+=\infty.\) Then 
     \begin{equation}\label{convoSAsympEq1}
         \PIp(x)\simo \frac{\MPsi(1-\alpha)}{1+\alpha}\,x\,\Pi\lbrb{\lbrb{\ln(1/ x),\infty}}.
     \end{equation}
\end{theorem}

\subsubsection{Integral equations}\label{SecintEq}
In this section we present integral equations for the law of \(\IPsi\). The first one is the most general and is a special case of \cite{BehLinRek21}, which considers exponential functionals of the type \(\int_0^\infty e^{-\xi_t}\D \eta_t\). To state the result, we need to introduce some notation. Recall the \LL triplet \((\gamma,\sigma,\Pi)\) and the killing term \(q \geq 0\) from \eqref{LLK}. Motivated by \cite[Prop.~3.1]{BehLinRek21}, set 
\begin{equation}\label{repar}
\widetilde\Pi := h(\Pi) + q\delta_{-1},
   \quad \text{and}
   \quad 
       \widetilde{\gamma} := -\gamma - \frac{\sigma^2}{2} 
       + \int_{-\ln 2}^{\infty} \lbrb{e^{-x}-1 + x\ind{|x|\leq 1}} \Pi(\D x) - q,
\end{equation}
where \(h(\Pi)\) is the image of \(\Pi\) under the mapping \(h(x) = e^{-x} - 1,\ x > -1\), and \(\delta_{-1}\) denotes the Dirac mass at \(-1\).
Define further
\begin{equation*}
    B(z) := \begin{cases}
        0 & \text{if } z = 1, \\[1ex]
       \widetilde{\Pi}\lbrb{\lbrb{\max\curly{z-1,1}, \infty}} & \text{if } z > 1,
    \end{cases}
\end{equation*}
and
\begin{equation*}
    S(z) := \begin{cases}
        \int_{-\infty}^{z-1}(z-y-1)\widetilde\Pi(\D y)\ind{\lbbrbb{-1,1}}(y) & \text{if } z \in \lbbrb{0,1}, \\[1ex]
        0 & \text{if } z = 1, \\[1ex]
        \int_{z-1}^\infty (y-z+1)\widetilde\Pi(\D y)\ind{\lbbrbb{-1,1}}(y) & \text{if } z > 1.
    \end{cases}
\end{equation*}

\begin{theorem}[Thm.~5.3 in \cite{BehLinRek21}]\label{intEQ}
      Let \(\Psi\) be the \LLK\ exponent of a potentially killed \LLP \(\xi\) such that \(\phi_-(0) > 0\) and set \(-\Psi(0) = q \geq 0\). Then there exists \(K \in \Rb\) such that
      \begin{equation}\label{intEQ_1}
          \begin{split}
              K\D z &= \frac{{\sigma}^2}{2} z^2 \Pbb{\IPsi \in \D z} 
              + \ind{z > 0} \IntOI x\, S\lbrb{\tfrac{z}{x}} \Pbb{\IPsi \in \D x}\,\D z \\
              &\quad \,+ \ind{z < 0} \int_{-\infty}^0 |x|\, S\lbrb{\tfrac{z}{x}} \Pbb{\IPsi \in \D x}\,\D z 
              - \widetilde{\gamma} \int_{0+}^z x \Pbb{\IPsi \in \D x}\,\D z \\
              &\quad \,- \int_{0+}^z \int_{0+}^t B\lbrb{\tfrac{t}{x}} \Pbb{\IPsi \in \D x} \, \D t\,\D z.
          \end{split}
      \end{equation}
\end{theorem}

\begin{remark}[\textbf{Historical remarks}]\label{intEQH}
    Equation~\eqref{intEQ_1} is itself a special case of the equation for the law of more general exponential functionals. Its immediate predecessor is given in \cite[Thm.~1]{KuzParSav2012}. However, that result comes with minor restrictions and an oversight in the proof, both of which are corrected in \cite{BehLinRek21}.
\end{remark}

Another general equation for the law of $\IPsi$ is provided by the next result,
which links the law of $\IPsi$ to the potential measure $U$ of $\xi$. Recall that
$U$ is defined by
\[
    U(\D x):=\IntOI \Pbb{\xi_t\in \D x}\,\D t.
\]
Then the following holds.

\begin{theorem}[Thm.~2.1 in \cite{ArRiv23}]\label{intEQ1}
    Let \(\Psi\) be the \LLK exponent of a potentially killed \LLP \(\xi\) such that \(\phi_-(0) > 0\). Let \(U\) denote the potential measure of \(\xi\), which is finite on every compact set since \(\xi\) is transient. Then
    \begin{equation}\label{intEQ1_eq}
        \PIPs{x} = \int_{\Rb} \pPs(x e^y)\, U(\D y), 
        \qquad \text{for almost every } x>0.
    \end{equation}
\end{theorem}
\begin{remark}[\textbf{Historical remarks}]\label{intEQ1H}
    The first integral equation in the literature appears to be in \cite[Prop. 2]{CarPetYor97}, where, for a conservative subordinator with a finite mean, it is shown that
    \[
        (1-\dbf_+x)\pPs(x)
        = \int_x^\infty \Pi\lbrb{\lbrb{\ln(y/x),\infty}}\pPs(y)\,\D y.
    \]
    The restriction to conservativeness and finite expectation was later removed in \cite[Thm.~2.1]{PaRiSch13}, where the equation becomes,
    for $x \in \lbrb{0,1/\dbf_+}$,
    \[
        (1-\dbf_+x)\pPs(x)
        = \int_x^\infty \Pi\lbrb{\lbrb{\ln(y/x),\infty}}\pPs(y)\,\D y
        - \Psi(0)\int_x^\infty \pPs(y)\,\D y.
    \]
\end{remark}

\section{Properties of Bernstein-gamma functions}
\label{sec: BG further}

\subsection{Analytic properties of Bernstein-gamma functions}
We first start with the proof of Theorem \ref{thm-WStir}.
\begin{proof}[Proof of Theorem \ref{thm-WStir}]
    Only the very last claim is not contained in the literature. From \eqref{receq} we obtain
    \begin{equation*}
        D_{a+1}
        = \liminfi{|b|}\frac{\ln\abs{\Wphi(a+ib)} + \ln\abs{\phi(a+ib)}}{|b|}
        = D_a,
    \end{equation*}
    since, by \cite[Prop.~3.1]{PatieSavov2018}, we have \(\ln\abs{\phi(z)}=\bo{\abs{\ln z}}\) as \(|z|\to\infty\). If \(a>1\), the same argument applies to \(D_{a-1}\). Hence, it suffices to consider \(D_{a'}\) with \(|a-a'|<1\) and \(a'>0\). In this case,
    \begin{equation*}
        \liminfi{|b|}\frac{\ln\abs{\Wphi(a'+ib)}}{|b|}
        = \liminfi{|b|}\lbrb{\frac{\ln\abs{\Wphi(a+ib)}}{|b|}+\frac{\ln\abs{\frac{\Wphi(a'+ib)}{\Wphi(a+ib)}}}{|b|}}.
    \end{equation*}
    Without loss of generality, assume \(b>0\). From \eqref{WStir} and \(\ln\abs{\phi(z)}=\bo{\abs{\ln z}}\), as \(|z|\to\infty\), it would follow that \(D_a=D_{a'}\) if
    \[
        0 = \limi{b}\frac{\abs{\Re\lbrb{L_\phi(a+ib-1) - L_\phi(a'+ib-1)}}}{b}.
    \]
    Choosing the contours \(1 \to a \to a+ib\) and \(1 \to a' \to a'+ib\), we obtain
    \begin{equation*}
        \begin{split}
           L_\phi(a+ib-1)-L_\phi(a'+ib-1)
           &= \int_{1}^a \ln \abs{\phi(u)}\,\D u - \int_{1}^{a'} \ln \abs{\phi(u)}\,\D u \\
           &\quad\, + i\left(\int_{0}^{b}\ln \abs{\phi(a+iu)}\,\D u - \int_{0}^{b}\ln \abs{\phi(a'+iu)}\,\D u\right) \\
           &\quad\, + \int_{0}^{b}\lbrb{\arg_0\phi(a+iu)-\arg_0\phi(a'+iu)}\,\D u.
        \end{split}
    \end{equation*}
    Thus it suffices to prove that
    \[
        0 = \limi{b}\frac{\abs{ \int_{0}^{b}\arg_0\lbrb{\frac{\phi(a+iu)}{\phi(a'+iu)}}\,\D u}}{b}
        = \limi{b}\frac{\abs{\int_{0}^{b}\arg_0\lbrb{1+\frac{\phi(a+iu)-\phi(a'+iu)}{\phi(a'+iu)}}\,\D u}}{b},
    \]
    where we used
    \(\arg_0\phi(a+iu)-\arg_0\phi(a'+iu) = \arg_0\lbrb{\frac{\phi(a+iu)}{\phi(a'+iu)}}\), since
    \(\arg_0\phi(a+iu),\arg_0\phi(a'+iu) \in (-\tfrac{\pi}{2},\tfrac{\pi}{2})\) as \(\phi(\CbOI)\subseteq \CbOI\), \(\Re(\phi(a+ib))\geq \phi(a)>0\), and \(\Re(\phi(a'+ib))\geq \phi(a')>0\); see \cite[Prop.~3.1 (9 and 10)]{PatieSavov2018}. 

    It remains to show that
    \[
        \limi{u}\frac{\phi(a+iu)-\phi(a'+iu)}{\phi(a'+iu)}=0.
    \]
    If \(\phi(\infty)<\infty\), this follows directly from \cite[Prop.~3.1(5)]{PatieSavov2018}. Otherwise,
    \[
        \abs{\phi(a+iu)-\phi(a'+iu)}
        = \abs{\int_{a}^{a'}\phi'(v+iu)\,\D v}
        \leq \abs{a-a'}\,\phi'\lbrb{\min\curly{a,a'}},
    \]
    which follows from \cite[Prop.~A.1(5)]{MinSav_2023}. This completes the proof.
\end{proof}

The next proposition provides an elementary lower bound for \(\abs{\Wp(z)}\).

\begin{proposition}\label{prop:lowerboundW}
    Let \(\phi\) be a Bernstein function. Then, for any \(a>0\) and \(z=a+ib\), 
    \begin{equation}\label{lowerboundW}
        \abs{\Wp(z)} \;\geq\; \frac{\Wp(a)}{\Gamma(a)} \abs{\Gamma(z)}.
    \end{equation}
\end{proposition}

\begin{proof}
The claim follows directly from the factorisation \eqref{FactE}, which shows that \(\Gamma(z)/\Wp(z)\) is the Mellin transform of \(I_\phi\), the exponential functional of a subordinator.
\end{proof}

The next result does not appear in the literature, but it provides uniform estimates along compact subsets of the real part of \BG functions, which will be useful in subsequent arguments. Throughout, we use the notation
\[
    f \asymp g \quad \text{to mean that} \quad 
    0 < \liminf_{x\to\infty} \frac{f(x)}{g(x)} 
    \leq \limsup_{x\to\infty} \frac{f(x)}{g(x)} < \infty.
\]

\begin{proposition}\label{UnifBGF}
Let \(\phi\) be a Bernstein function with associated \BG function \(\Wphi\), and let \((c,d)\subset(\baph,\infty)\). Then, as \(|b| \to \infty\), 
\begin{equation}\label{unif}
  \sup_{a,a'\in(c,d)}
    \left|\frac{\Wphi(a+ib)}{\Wphi(a'+ib)}\right|
  \;\asymp\;
  \sup_{a,a'\in(c,d)}
    \exp\!\lbrb{\int_{1+a}^{1+a'} \ln\bigl|\phi(u+ib)\bigr|\,\D u}.
\end{equation}
Moreover,
\[
  \sup_{u\in(c,d)} \ln\bigl|\phi(u+ib)\bigr|
  =
  \begin{cases}
    \bo{\ln|b|}, & \text{if } \dbf>0,\\[0.5ex]
    \so{\ln|b|}, & \text{otherwise}.
  \end{cases}
\]

Finally, the same uniform estimate extends as follows:
\begin{enumerate}
  \item if \(\liminfi{|b|}\phi(ib) > 0\), then \eqref{unif} holds for any \([c,d)\subset[0,\infty)\);
  \item if, in addition, \(\baph<0\), \(\abs{\phi(\baph)}<\infty\), and \(\liminfi{|b|}\phi(\baph+ib)>0\), then \eqref{unif} holds for any \([c,d)\subset[\baph,\infty)\).
\end{enumerate}
\end{proposition}

\begin{proof}
Let \([c, d) \subset [1, \infty)\) and \(|b|>0\). We use \eqref{repW} together with the uniform bound on \(E_\phi\) from Theorem~\ref{repWTh}, and choose the contours \(1 \to 1+a+ib\) and \(1 \to 1+a+ib \to 1+a'+ib\) for \(L_{\phi}(a+ib)\) and \(L_\phi(a'+ib)\), respectively, to obtain
\begin{equation*}
     \sup_{a,a' \in (c,d)} \left| \frac{\Wphi(a+ib)}{\Wphi(a'+ib)} \right| 
     \asymp  
     \sup_{a,a' \in (c,d)} 
     \left| \frac{\phi(a'+ib)}{\phi(a+ib)} \sqrt{ \frac{\phi(a'+1+ib)}{\phi(a+1+ib)} } \right| 
     \exp \left( \int_{1+a}^{1+a'} \ln|\phi(u+ib)| \, \D u \right).
\end{equation*}
For any \((c, d) \subseteq (\baph, \infty)\) we note that
\begin{equation}\label{uniPhi}
    \begin{split}
        \sup_{a,a' \in (c,d)} \left| \frac{\phi(a'+ib)}{\phi(a+ib)} \right|
        &= \sup_{a,a' \in (c,d)} \left| 1 + \frac{\phi(a'+ib) - \phi(a+ib)}{\phi(a+ib)} \right| \\
        &\leq \sup_{a,a' \in (c,d)} \left[ 1 + \left| \frac{\int_a^{a'} \phi'(u+ib) \, \D u}{\phi(a+ib)} \right| \right] \\
        &\leq \sup_{a,a' \in (c,d)} \left[ 1 + \frac{|a - a'| \phi'(\min(a,a'))}{\phi(a)} \right] \\
        &\leq 1 + \frac{(d-c)\phi'(c)}{\phi(c)} \leq \frac{d}{c},
    \end{split}
\end{equation}
where we used that \(\abs{\phi'(z)} \leq \dbf + \int_0^\infty e^{-\Re(z)y} y \mu(\D y) = \phi'(\Re(z))\), and the facts that 
\(\abs{\phi(a+ib)} \geq \Re(\phi(a+ib)) > \phi(a)\), and \(\phi'(x)/\phi(x)\leq 1/x\); see \cite[Prop.~A.1 (1,5)]{MinSav_2023} for details. By symmetry, the infimum of the ratio of Bernstein functions is bounded from below by \(c/d\).  

The behaviour of \(\ln|\phi(u+ib)|\) follows from \cite[Prop.~3.1(4)]{PatieSavov2018}, since \(|\phi(z)|=\dbf|z|(1+\so{1})\) uniformly on \(\Cb_{(\baph,\infty)}\). To extend \eqref{unif} to \((c, d) \subseteq (\baph, \infty)\), we can apply \eqref{receq} and \eqref{uniPhi} to transfer the range to \((c-\baph+2,d-\baph+2)\subset (1,\infty)\).  

Finally, the last cases, i.e. when \(\liminfi{|b|}\phi(ib)>0\) or \(\liminfi{|b|}\phi(\baph+ib)>0\), are handled analogously via \eqref{receq} and \eqref{uniPhi}.
\end{proof}

Finally, we collect the most important complex-analytical properties of $\Wphi$.

\begin{theorem}[Thm.~4.1 in \cite{PatieSavov2018}]\label{W-prop}
  Let \(\phi\) be a Bernstein function and let \(\Wphi\) be its associated \BG function.  Then
  \begin{enumerate}
    \item 
      \(\Wphi\in\Att_{\lbrb{\baph,\infty}}\cap\Mtt_{\lbrb{\aph,\infty}}\) 
      and is zero‐free on \(\Cb_{\lbrb{\aph,\infty}}\).
    \item 
      If \(\phi(0)>0\), then moreover 
      \(\Wphi\in\Att_{\lbbrb{0,\infty}}\) 
      and is zero‐free at least on \(\Cb_{\lbbrb{0,\infty}}\).
    \item 
      If \(\phi(0)=0\) and 
      \(\Zc_0(\phi):=\{\,\zeta\in i\Rb:\phi(\zeta)=0\}\), 
      then \(\Wphi\) extends continuously to \(i\Rb\setminus\Zc_0(\phi)\), and for each \(\zeta\in\Zc_0(\phi)\)
      \[
        \lim_{\substack{z\to\zeta\\\Re(z)\ge0}}
          \phi(z)\,\Wphi(z)
        \;=\;\Wphi(\zeta+1)\,.
      \]
    \item 
      If \(\phi(0)=0\) and \(\lvert\phi''(0^+)\rvert<\infty\), then
      \[
        z\mapsto\Wphi(z)-\frac{1}{\phi'(0+)z}
        \text{   extends continuously on \(\,(i\Rb\setminus\Zc_0(\phi))\cup\{0\}\).}
      \]
    \item 
      Finally, for any constant \(c>0\),
      \[
        W_{c\phi}(z)=c^{\,z-1}\,\Wphi(z).
      \]
  \end{enumerate}
\end{theorem}

\begin{remark}[\textbf{\textit{Further comments}}]\label{W-propF}
    Theorem~4.1 in \cite{PatieSavov2018} provides additional complex-analytic properties of \(\Wphi\). The most essential among them concerns the case \(\uph \in [\aph,0)\), which is linked to Cramér's condition and plays a role in the proof of \cite[Thm.~2.11]{PatieSavov2018}. When the \LL measure of \(\phi\) is supported on a lattice and \(\phi(0) = \dbf = 0\), then \(|\Wphi|\) is periodic on \(\Cb_{\lbrb{0,\infty}}\); see again \cite[Thm.~4.1]{PatieSavov2018}. It may also happen that \(\Wphi\) admits an analytic continuation beyond \(\Cb_{({\aph,\infty})}\). For instance, if \(\phi(z) = \log(z+1)\), then \(\aph = -1\), yet \(\Wphi\) is analytic on \(\Cb \setminus (-\infty,0]\).
\end{remark}

\subsection{Transformation properties of Bernstein-gamma functions}
Here we present a couple of transformation properties of \BG functions.

First, recall the definition of \(\aph,\uph\), see~\eqref{aph}. For \(\beta>\aph\) we have
\begin{equation}\label{phib}
\phi_\beta(z):=\phi(z+\beta)=\phi(\beta)+\dbf z +\int_0^\infty (1-e^{-zy})e^{-\beta y}\mu(\D y), \qquad \Re(z)\geq 0.    
\end{equation}
The identity also holds for $\beta=\aph$, provided that $\abs{\phi(\aph)}<\infty$. Moreover, for all $\beta>\uph$ (and also for $\beta=\uph$ when $\phi(\uph)=0$), the function $\phi_\beta$ is a Bernstein function. We thus obtain the following elementary fact.

\begin{proposition}\label{prop:Wphib}
  Let $\phi$ be a Bernstein function. For $\beta>\uph$, we have
  \begin{equation}\label{Wphib}
     W_{\phi_\beta}(z) \;=\; \frac{\Wp(z+\beta)}{\Wp(1+\beta)}, \qquad \Re(z)>0,
  \end{equation}
  and this defines the \BG function associated with~$\phi_\beta$. The same relation also holds for $\beta=\uph$, provided that $\phi(\uph)=0$.
\end{proposition}

Next, we consider the transformation under the so-called $T_\beta$--transform of $\phi$, whose most general form is discussed in~\cite{ChKypPat2021} and whose usefulness is illustrated in~\cite{PatieSavov2013}. Let \(\phi\) be a Bernstein function and \(\beta>0\). From~\eqref{BG} we obtain
\begin{equation*}
\begin{split}
    (T_\beta\phi)(z) 
    &:= \frac{z}{z+\beta}\phi_\beta(z) 
      = \frac{z}{z+\beta}\phi(0) + \dbf z 
        + z \int_0^\infty e^{-zy} e^{-\beta y}\bar{\mu}(y)\,\D y \\
    &\hspace{0.3em}= \dbf z 
        + z \int_0^\infty e^{-zy} e^{-\beta y}\bigl(\bar{\mu}(y)+\phi(0)\bigr)\,\D y,
        \qquad \Re(z)\geq 0,
\end{split}
\end{equation*}
which is a Bernstein function, since 
\[
   \bar{\mu}_\beta(y) := e^{-\beta y}\bigl(\bar{\mu}(y)+\phi(0)\bigr)
\]
is decreasing and integrable, and thus represents the tail of a measure $\mu_\beta$ defining a Bernstein function corresponding to an unkilled subordinator. By uniqueness of the \BG function, we obtain:

\begin{proposition}\label{prop:Wphib1}
  Let $\phi$ be a Bernstein function. For $\beta>0$, we have
  \begin{equation}\label{Wphib2}
     W_{T_\beta\phi}(z) \;=\; \Gamma(1+\beta)\frac{\Gamma(z)}{\Gamma(z+\beta)}\Wp(z), \qquad \Re(z)>0,
  \end{equation}
  and this defines the \BG function associated with~$T_\beta \phi$. 
\end{proposition}

Finally, Proposition~\ref{prop:Wphib1} is a manifestation of the more general principle that if 
$\phi, \phi_1, \phi_2$ are Bernstein functions with $\phi = \phi_1 \phi_2$, then
\begin{equation}\label{prodphi}
    \Wp(z) \;=\; W_{\phi_1}(z)\,W_{\phi_2}(z), \qquad \Re(z)>0.
\end{equation}
We refer to \cite[Chapter 4.2]{PatSav21}
for a discussion of such triples of Bernstein functions.
\subsection{Bernstein-gamma functions of Wiener-Hopf factors}\label{subsec:BGWH}
We recall that when \(\Psi\) is the \LLK exponent of a potentially killed \LL process, the Wiener--Hopf factorisation \eqref{WHf} holds, namely 
\(
   \Psi(z) = -\phi_+(-z)\phi_-(z).
\)
In this case, there is an alternative expression for the leading term \(A_\phi\) in the Stirling asymptotics of \(\abs{\Wp(z)}\); see \eqref{Aphi}. 

\begin{proposition}[Lem.~A.2 in \cite{PatieSavov2018}]\label{WStirWH}
  Let \(\Psi\) be the \LLK exponent of a potentially killed \LLP \(\xi\) with Wiener--Hopf factors \(\phi_\pm\). Then, for any \(z=\ab\in\CbOI\), 
  \begin{equation}\label{WStirWH_1}
      A_{\phi_\pm}(\ab)
      = \int_0^\infty \int_{[0,\infty)} \frac{1-\cos(bx)}{x}\, e^{-ax}\,\Pbb{\pm \xi_t\in \D x}\,\frac{\D t}{t}
      = \int_{[0,\infty)} \frac{1-\cos(bx)}{x}\, e^{-ax}\,H_\pm(\D x),
  \end{equation}
  where \(H_\pm(\D x) := \int_0^\infty \Pbb{\pm \xi_t\in \D x}\,\frac{\D t}{t}\) are the so-called harmonic measures of \(\xi\). Finally, if we vary the killing rate \(-\Psi(0)=q\geq 0\) then \(A_{\phi_\pm}(\ab)\) are non-increasing as functions of \(q\).
\end{proposition}

\begin{remark}[\textbf{Historical remarks}]\label{WStirWHH}
   In \cite[Lem.~A.2]{PatieSavov2018}, the expression in \eqref{WStirWH_1} is written with \(e^{\Psi(0)t}\Pbb{\pm\xi^\sharp_t\in \D x}\) for the identical \(\Pbb{\pm \xi_t\in \D x}\), where \(\xi^\sharp\) is the \LLP pertaining to the \LLK exponent \(\Psi^\sharp (z)=\Psi(z)-\Psi(0)\).  The final claim of the proposition is a direct consequence of this observation.  
\end{remark}

\begin{remark}[\textbf{\textit{Further comments}}]\label{WStirWHF}
     The monotonicity in \(a\) is clear from \eqref{WStirWH_1}. Changing variables, we obtain
     \begin{equation*}
      A_{\phi_\pm}(\ab)
      = b \int_{[0,\infty)} \frac{1-\cos(x)}{x}\, e^{-\tfrac{a}{b}x}\, H_\pm\!\left(\tfrac{\D x}{b}\right),
     \end{equation*}
     and therefore, from equation \eqref{WStirAbs} of Proposition~\ref{prop-WStir}, we derive, as \(b \to \infty\),
     \[
        \ln\abs{W_{\phi_\pm}(\ab)}
        \;\simi\; -A_{\phi_\pm}(a+ib)
        = -b \int_{[0,\infty)} \frac{1-\cos(x)}{x}\, e^{-\tfrac{a}{b}x}\, H_\pm\!\left(\tfrac{\D x}{b}\right),
     \]
     and from \eqref{geomDecay1},
     \[
        \frac{1}{b}\int_0^{b} \arg_0 \phi(a+iu)\,\D u
        \;=\; \int_{[0,\infty)} \frac{1-\cos(x)}{x}\, e^{-\tfrac{a}{b}x}\, H_\pm\!\left(\tfrac{\D x}{b}\right)
        \;\simi\; \int_{[0,\infty)} \frac{1-\cos(x)}{x}\, H_\pm\!\left(\tfrac{\D x}{b}\right).
     \]
     That is, the constant of the linear decay of \(\ln\abs{W_{\phi_\pm}(\ab)}\) depends on the harmonic measures at \(0\).  
\end{remark}

    \section*{Acknowledgments}
The authors would like to thank the anonymous referee for the careful reading
of the manuscript and for constructive comments that helped improve the
presentation. The first author acknowledges the affiliation with the Faculty of
Mathematics and Informatics, Sofia University ``St. Kliment Ohridski'', and
the second author acknowledges the affiliation with the Institute of Mathematics
and Informatics, Bulgarian Academy of Sciences.

This study was financed by the European Union’s NextGenerationEU, through the National Recovery and Resilience Plan of the Republic of Bulgaria, project No. BG-RRP-2.004-0008.

Part of the work of M.M. was carried out after the end of his involvement in this project, during his postdoctoral fellowship at the UZH, supported by SNSF SCIEX programme (IZSF-0-235789).

\appendix

\section{Examples of Bernstein-gamma functions}
\label{sec: BG examples}    
The first seven examples are from
 \cite{BarPat21}. Next, are based on Bernstein functions from 
 \cite[Ch. 16]{SchVon2020}. We denote by \(G(\cdot,\cdot)\) the Barnes-gamma function, that is the solution of
 \[G(z+1, a)=\Gamma(z/a)G(z,a),
 \quad \text{with}
 \quad G(1,a)=1,
 \]
 for $a>0$, see \cite[Sec.~2]{LoePatSav19}. 

\begin{table}[ht]
  \centering\small
  \caption{Examples of Bernstein functions $\phi$, their $W_\phi$ with characteristics from \eqref{aph}, and parameter ranges.}
  \begin{tabular}{@{}%
      p{0.15\textwidth}
      p{0.30\textwidth}
      p{0.05\textwidth}
      p{0.05\textwidth}
      p{0.05\textwidth}
      p{0.25\textwidth}@{}
    }
    \toprule
    Bernstein \(\phi(z)\) &
      Bernstein-gamma \(W_\phi(z+1)\) &
      $\aph$& $\uph$&$\baph$&
      \textbf{Notes} \\
        \toprule
      \(z+q\) &
      \(\displaystyle \frac{\Gamma(z+1+q)}{\Gamma(1+q)}\) &$-\infty$&$-q$&$-q$&
      \(q\ge0\)\\[1ex]
      \midrule
    \((z+q)^\alpha\) &
      \(\displaystyle \frac{\Gamma^\alpha(z+1+q)}{\Gamma^\alpha(1+q)}\) &$-q$&$-\infty$&$-q$&
      \(q\ge0, \; \alpha\in(0,1)\)\\[1ex]
      \midrule
    \(\displaystyle\frac{z+1-a}{z+b}\) &
      \(\displaystyle \frac{\Gamma(b+1)\Gamma(z+2-a)}{\Gamma(z+b+1)\Gamma(2-a)}\) & $-b$&$a-1$&$a-1$&
      \(a \in (0,1), \;b\geq1-a\) \\[1ex]
      \midrule

    \(\displaystyle\frac{\Gamma(a z + b)}{\Gamma(a(z-1) + b)}\) &
      \(\displaystyle \frac{\Gamma(a z + b)}{\Gamma(b)}\)& $1-\frac{b}{a}$&$-\infty$& $1-\frac{b}{a}$&
      \(a\in(0,1),\; b\ge a\)
      \\[1ex]
      
      \midrule
          \(\displaystyle\frac{\Gamma(  az+a)}{z\Gamma(az)}\) &
      \(\displaystyle \frac{\Gamma(az+a)}{\Gamma(a)\Gamma(z+1)}\) & $0$&$-\infty$
          &$0$&
      \(a \in (1,2)\) \\[1ex]
      \midrule

    \(\displaystyle \frac{\Gamma(a z + a\rho)}{\Gamma(a\rho)}\) &
      \(\displaystyle \frac{G(1,\,1/a)\,G(z+1+\rho,\,1/a)}{G(1+\rho,\,1/a)\,G(z+1,\,1/a)}\) &\(-\frac{\rho}{a}\) & 
\(-\infty\) & 
\(-\frac{\rho}{a}\) & 
      \(a\in[1,2],\; \rho\in[0,1/a]\) \\[1ex]
            \midrule

    \(\displaystyle \frac{\Gamma(z + a + b)}{\Gamma(z+a)}\) &
      \(\displaystyle \frac{G(z+a+b+1,\,1)\,G(a+1,\,1)}{G(z+a+1,\,1)\,G(a+b+1,\,1)}\) &\(-{a-b}\) & 
\(-a\) & 
\(-a\) & 
      \(a\geq0, \, b\in(0,1)\) \\[1ex]
      \midrule

    \(1 - q^{z+b}\) &
     Pochhammer's \((q^b,q)_z\) &\(-\infty\) & 
\(-b\) & 
\(-b\) & 
      \(q\in(0,1),\; b\ge0\) \\
            \midrule

    \(\displaystyle\frac{z}{z+a}\) &
     $\displaystyle\frac{\Gamma(z+1)\Gamma(a)}{\Gamma(z+1+a)}$&$-a$&  
$0$ & 
\(-a\) & 
      \(a>0\) \\
            \midrule

    \(\displaystyle\frac{z}{(z+a)^\alpha}\) &
   \( \displaystyle\frac{\Gamma(z+1)\Gamma^\alpha(a)}{\Gamma^\alpha(z+1+a)}\) &\(-a\) & 
\(0\) & 
\(0\) & 
      \(a>0, \,\alpha\in(0,1)\) \\
            \midrule

    \(\displaystyle\frac{z^\alpha}{(1+z)^\alpha}\) &
     $\displaystyle\frac{\Gamma^\alpha(z+1)}{\Gamma^\alpha(z+2)}$ &\(0\) & 
\(0\) & 
\(0\) & 
      \(\alpha\in(0,1)\) \\
                  \midrule
$\displaystyle \frac{1-q^z}{1-q} $
&$\Gamma_q(z)$&$-\infty$&
$0$
&
$0$
&\(q\in(0,1)\)\\
    \bottomrule
  \end{tabular}
  \label{tab:BG-examples}
\end{table}

\clearpage
\begin{landscape}
\section{Relevant parameters for some Lévy processes}
\label{sec: aphi examples}    

\begin{table}[ht]
  \centering\small
  \caption{Examples for $  \ak_{\pm}^\Psi,\, \uk_{\pm}^\Psi,$ and
  $\bar{a}_{\pm}^\Psi$ from
      \eqref{aPsi}}
  \begin{tabular}{@{}%
      p{0.25\textwidth}
      p{0.15\textwidth}
      p{0.15\textwidth}
      p{0.15\textwidth}
      p{0.2\textwidth}
            p{0.05\textwidth}
      p{0.05\textwidth}
      p{0.2\textwidth}
    }
    \toprule
     \(\Psi(z)\) &
      $\ak_{+}^\Psi$&
      $\ak_{-}^\Psi$&
      $\uk_{+}^\Psi$&
      $\uk_{-}^\Psi$&
      $\bar{a}_{+}^\Psi$&
      $\bar{a}_{-}^\Psi$
      &
      \textbf{Notes} \\
        \toprule
        $\displaystyle-q+\mu z+ \frac{1}{2}
        \sigma^2 z^2$
        &$\infty$
        &$-\infty$
&$\scriptstyle\frac{-\mu+\sqrt{\mu^2+2\sigma^2q}}{\sigma^2}$
&$\scriptstyle\frac{-\mu-\sqrt{\mu^2+2\sigma^2q}}{\sigma^2}$
&$\uk_{+}^\Psi$
   &$\uk_{-}^\Psi$     
&
$\sigma,q\geq0$, $\mu\in\Rb$\\
\midrule
        $\displaystyle-q-c\log_0\lbrb{1-\frac{z}{\theta}}$
        &$\theta$
        &$-\infty$
&$\displaystyle \theta\lbrb{1-e^{-q/c}}$
&$-\infty$
&$\displaystyle \uk_+^\Psi$
   &$-\infty$     
&
$q,\,c,\,\theta\geq0$\\
\midrule
        $\displaystyle -q - s\lbrb{\sqrt{b^2 - 2z} - b}$
        & $\displaystyle \frac{b^2}{2}$
        & $-\infty$
        & $\displaystyle \frac{bq}{s} - \frac{q^2}{2s^2}$
        & $-\infty$
        & $\ast$
        & $-\infty$
        & \(b,s>0,\;q \le sb\) \\
\midrule
        $\displaystyle -q - s\lbrb{\sqrt{b^2 - 2z} - b}$
        & $\displaystyle \frac{b^2}{2}$
        & $-\infty$
        & $\displaystyle \infty$
        & $-\infty$
        & $\displaystyle \frac{b^2}{2}$
        & $\displaystyle -\infty$
        & \(b,s>0,\;q > sb\) \\
\midrule
       $\textstyle-\beta\log_0\lbrb{1-\frac{cz}{\alpha}
       -\frac{\sigma^2z^2}{2\alpha}}$
        & $\textstyle \frac{-c+\sqrt{c^2+2\alpha\sigma^2}}{\sigma^2}$
        & $\textstyle\frac{-c-\sqrt{c^2+2\alpha\sigma^2}}{\sigma^2}$
        & $0$
        & $\displaystyle-\frac{c}{\sigma^2}$
        & $\displaystyle \uk_+^\Psi$
        & $\displaystyle \uk_-^\Psi$
        & Variance gamma 
        \cite{Madan-Seneta-1990},
        $\alpha,\,
        \beta,\,
        \sigma>0$; $c\in\Rb$\\
\midrule
       $ c\lbrb{\frac{\Gamma(z+\gamma+\alpha)}{\Gamma(z+\gamma)}-\frac{\Gamma(\gamma+\alpha)}{\Gamma(\gamma)}}$
        & $\displaystyle\infty$
        & $-\gamma-\alpha$
        & $0$
        & $\scriptstyle\uk_-^\Psi\in(-\gamma-\alpha, -\gamma-1)$
        & $0$
        & $\displaystyle \uk_-^\Psi$
        &
        \cite{Patie-2009},
        $\scriptstyle\alpha \in (1,2], \,
        \gamma>-\alpha$\\
        \midrule
    \bottomrule
  \end{tabular}
  \label{tab:aphi-examples}
\end{table}

\begin{table}[ht]
  \centering\small
  \caption{Examples for $  \ak_{\pm}^\Psi,\, \uk_{\pm}^\Psi,$ and
  $\bar{a}_{\pm}^\Psi$ from
  \eqref{aPsi}}
  \begin{tabular}{@{}%
      p{0.35\textwidth}
      p{0.27\textwidth}
      p{0.25\textwidth}
      p{0.05\textwidth}
      p{0.05\textwidth}
            p{0.05\textwidth}
      p{0.05\textwidth}
      p{0.24\textwidth}
    }
    \toprule
     Process &
      $\ak_{+}^\Psi$&
      $\ak_{-}^\Psi$&
      $\uk_{+}^\Psi$&
      $\uk_{-}^\Psi$&
      $\bar{a}_{+}^\Psi$&
      $\bar{a}_{-}^\Psi$
      &
      \textbf{Notes} \\
        \toprule
       \footnotesize$\Ebb{e^{\lambda\xi_1}}$ is not finite for all $\lambda\neq0$ 
        &$0$
        &$0$
&$0$
&$0$
&$0$
   &$0$     
&
\footnotesize two-sided heavy-tailed\\
\midrule
        subordinator&
        \footnotesize
         $\sup \lambda$ with
        $\Ebb{e^{\lambda\xi_1}}<\infty$
        &$-\infty$
   &    
   $\ast$
&$-\infty$
&$\ast$
&$-\infty$
&$\scriptstyle \Psi(z)=-\phi_+(-z),
\phi_-\equiv 1$
\\
   \midrule
       $\xi$ drifting to $\infty$ 
        &  \footnotesize
        $\sup \lambda$ with
        $\Ebb{e^{\lambda\xi_1}}<\infty$
        & \footnotesize$\inf \lambda$ with
        $\Ebb{e^{\lambda\xi_1}}<\infty$
        & $0$
        & $\ast$
        & $0$
        & $\ast$
        &
        $\phi_-(0)>0$\\
\midrule
hypergeometric,  Section \ref{hyperLP}
        &  $1-\beta+\gamma$
        & $-\hat{\beta}-\hat{\gamma}$
        & \footnotesize$1-\beta$
        & $-\hat{\beta}$
        &\footnotesize $1-\beta$
        & $-\hat{\beta}$
        &
       \\
        \midrule

       meromorphic,  Section \ref{merLP} 
        &  $\rho_1$
        & $-\hat{\rho}_1$
        & $\ast$
        & $\ast$
        & $\ast$
        & $\ast$
        &
       see
       \cite[Cor. 2]{KuKyPa12}\\
\midrule

    \bottomrule
  \end{tabular}
  \label{tab:apsi-examples}
\end{table}

\end{landscape}
\clearpage
\section*{Index of notation}
\addcontentsline{toc}{section}{Index of notation}

The notation below is organized thematically. Within each group, entries are ordered by the main symbol whenever possible.
\subsection*{General}

\begin{longtable}{@{}p{0.24\textwidth}p{0.70\textwidth}@{}}
$\Att_{(a,b)},\,\Mtt_{(a,b)}$
& analytic and, respectively, meromorphic functions on $\Cb_{(a,b)}$ \\[3pt]

$\Cb_a$
& $\curly{z\in\Cb:\Re(z)=a}$ \\[3pt]

$\Cb_{(a,b)}$
& $\curly{z\in\Cb:\Re(z)\in\lbrb{a,b}}$ \\[3pt]

$\log_0,\,\log$
& principal and, respectively, some branch of the complex logarithm
\end{longtable}

\subsection*{Lévy processes and exponential functionals}

\begin{longtable}{@{}p{0.24\textwidth}p{0.70\textwidth}@{}}
$\dbf_\pm$
& drift terms of $\phi_\pm$ \\[3pt]

$F_\Psi,\,f_\Psi$
& c.d.f. and density of $I_\Psi$, Section~\ref{SmoothSec} \\[3pt]

$I_\Psi$
& exponential functional of $\xi$, \eqref{ExpFunc} \\[3pt]

$\MPsi(z)$
& Mellin transform of $I_\Psi$, \eqref{MT} \\[3pt]

$\NPs$
& decay of $\MPsi$ along complex lines, Theorem~\ref{MDecay}, \eqref{Npsi} \\[3pt]

$\phi_\pm$
& Wiener--Hopf factors from $\Psi(z)=-\phi_+(-z)\phi_-(z)$, \eqref{WHf} \\[3pt]

$\Psi$
& Lévy--Khintchine exponent of a Lévy process $\xi$, \eqref{LLK}
\end{longtable}

\subsection*{Bernstein--gamma functions}

\begin{longtable}{@{}p{0.24\textwidth}p{0.70\textwidth}@{}}
$\ak_\pm^\Psi,\,\bar{\ak}_\pm^\Psi,\,\uk_\pm^\Psi$
& analytical characteristics of the Lévy--Khintchine exponent $\Psi$, see \eqref{aPsi} \\[3pt]

$\aph$
& $\inf\curly{u<0:\phi\in\Att_{(u,\infty)}}\in\lbbrbb{-\infty,0}$, \eqref{aph} \\[3pt]

$\Aph$
& component for the Stirling-type asymptotics for $W_\phi$, \eqref{Aphi} \\[3pt]

$\baph$
& $\max\curly{\aph,\uph}\in\lbbrbb{-\infty,0}$, \eqref{aph} \\[3pt]

$c^\Psi$
& $-\ak^\Psi_+\ind{\uk^\Psi_+=0}\in[-\infty,0]$, Theorem~\ref{thm: MPsi} \\[3pt]

$L_\phi,\,T_\phi,\,E_\phi$
& components for the Stirling-type asymptotics for $W_\phi$, \eqref{terms} \\[3pt]

$\uph$
& $\sup\curly{u\in\lbbrbb{\aph,0}:\phi(u)=0}\in\lbbrbb{-\infty,0}$, \eqref{aph} \\[3pt]

$W_\phi$
& Bernstein--gamma function of $\phi$, see Theorem~\ref{repWTh}, \eqref{repW}
\end{longtable}

\clearpage

\bibliographystyle{alphadoi}

\end{document}